\documentclass[leqno]{siamltex}
\usepackage{amsmath,amssymb}
\usepackage{graphicx}
\usepackage{enumerate}           
\usepackage{url}
\usepackage{color}
\usepackage{tikz}
\usepackage{subfig}
\usetikzlibrary{arrows,intersections}
\usepackage{pgfplots}
\pgfplotsset{compat=1.12}

\usepackage{caption}

\def \dis{\displaystyle}

\def \R{\mathbb{R}} 

\def \N{\mathbb{N}}

\def \N0{\mathbb{N}_0}
\def \inN{\in \mathbb{N}}

\def \a{\alpha}

\def \d{\delta}
\def \e{\varepsilon}

\def \W{\Omega}
\def \phi{\varphi}

\def \12{\dis\frac{1}{2}}

\def \1{\mathbbm{1}}

\def \<{\left<}
\def \>{\right>}

\def \mA{\mathcal{A}}
\def \mAe{\mathcal{A}^\e}
\def \mAze{\mathcal{A}_0^\e}
\def \Th{\mathcal{T}_h}

\def \mL{\mathcal{L}}
\def \mK{\mathcal{K}}
\def \mKinv{\mathcal{K}^{-1}}
\def \mLe{\mathcal{L}^{\varepsilon}}
\def \mLhe{\mathcal{L}_h^{\varepsilon}}
\def \mLhes{(\mathcal{L}_h^{\varepsilon})^*}
\def \mLz{\mathcal{L}_0}
\def \mLze{\mathcal{L}_0^{\varepsilon}}
\def \mLzhe{\mathcal{L}_{0,h}^{\varepsilon}}
\def \grad{\nabla}
\def \lss{\lesssim}
\def \gss{\gtrsim}
\def \wto{\rightharpoonup}
\def \ue{u^\e}

\let\div\undefined
\DeclareMathOperator*{\div}{div}
\def \dx[#1]{\ensuremath{\operatorname{d}\!{#1}}}

\newtheorem{defn}{Definition}
\numberwithin{defn}{section}
\newtheorem{remark}{Remark}
\numberwithin{remark}{section}

\begin{document}
	
\title{Analysis of the Vanishing Moment Method and its Finite Element Approximations for Second-order Linear Elliptic PDEs in Non-divergence Form\thanks{This work was partially supported by the NSF grant
		 DMS-1620168.}  This paper  is dedicated to  Professor Roland Glowinski on the occasion of his eightieth birthday.}

\author{Xiaobing Feng\thanks{Department of Mathematics, The University of Tennessee, Knoxville, TN 37996, U.S.A. (xfeng@math.utk.edu).}
	\and
	Thomas Lewis\thanks{Department of Mathematics and Statistics, The University of North Carolina at Greensboro, Greensboro, NC 27412, U.S.A. 
	(tllewis3@uncg.edu).}
	\and 
	Stefan Schnake\thanks{Department of Mathematics, The University of Oklahoma, Norman, OK 73019, U.S.A. (sschnake@ou.edu).}
}

\maketitle
\date{\today}

\begin{abstract}
This paper is concerned with continuous and discrete approximations of $W^{2,p}$ strong solutions of second-order linear elliptic partial differential 
equations (PDEs) in non-divergence form.  The continuous approximation of these equations is achieved through the Vanishing Moment Method (VMM) 
which adds a small biharmonic term to the PDE.  The structure of the new fourth-order PDE is a natural fit 
for Galerkin-type methods unlike the original second order equation since the highest order term is in divergence form.  
The well-posedness of the weak form of the perturbed fourth order equation is shown 
as well as error estimates for approximating the strong solution of the original second-order PDE.  
A $C^1$ finite element method is then proposed for the fourth order equation, 
and its existence and uniqueness of solutions as well as optimal error estimates in the $H^2$ norm are shown.  Lastly, numerical tests are given to show 
the validity of the method. 
\end{abstract}



\pagestyle{myheadings}
\thispagestyle{plain}
\markboth{XIAOBING FENG AND THOMAS LEWIS AND STEFAN SCHNAKE}{THE VMM FOR NON-DIVESRGENCE FORM ELLIPTIC PDES}
 
\section{Introduction} \label{section1}

In this paper, we propose $C^1$ finite element approximations of the following second-order linear elliptic PDE in non-divergence form:
\begin{align} \label{eqn1.1} \tag{$P$}
\begin{split} 
\mL u:= -A:D^2u+b\cdot\grad u + cu &= f \qquad \text{ in } \W, \\ 
u &= 0 \qquad \text{ on } \partial\W,
\end{split}
\end{align} \addtocounter{equation}{1}
where 
$\W$ is an open and bounded domain in $\R^n$ and $\partial\W$ denotes its boundary.  These non-divergence form PDEs have several applications including game theory, stochastic optimal control, and mathematical finance \cite{SP:Flem}.  Moreover, non-divergence PDEs explicitly appear in several second-order fully nonlinear 
PDEs such as Hamilton-Jacobi-Bellman and Issac's equations as well as in the linearization of the Monge-Amp\`ere equation \cite{Crandall_Ishii_Lions,AM:CG}.

When the coefficient matrix $A$ is not smooth, \eqref{eqn1.1} cannot be written in divergence form.  Thus, any standard notion of weak solutions to \eqref{eqn1.1} must be abandoned, and, indeed, the PDE theory respects this observation and seeks well-posedness of these equations in a stronger sense.  There have been three main theories for the existence and uniqueness of these equations.  First, Schauder (or classical) theory seeks solutions in $C^{2,\a}(\W)$ where $a_{ij},b_i,c,f\in C^{\a}(\W)$ and $\partial\W\in C^{2,\a}$.  Second, strong solution theory seeks solutions in $W^{2,p}(\W)\cap W_0^{1,p}(\W)$ that satisfy the PDE almost everywhere. 
There have been two frameworks that guarantee unique strong solutions.  The first requires $a_{ij}\in C(\overline{\W})$, $b_i\in L^{\infty}(\W),$ $c\in L^{\infty}(\W),$ and  $f\in L^p(\W),$ with $1< p < \infty$ and $\partial\W\in C^{1,1}$ while the second requires $p=2$, $f\in L^2(\W)$, $\partial\W$ convex, and $a_{ij}\in L^{\infty}(\W)$ where the matrix $A$ satisfies the C\'{o}rdes condition \cite{Sp:DT,SmearsSuli}.  The last theory, called viscosity solution theory, seeks solutions in $C^0(\overline{\W})$ given $a_{ij},b_i,c\in L^{\infty}(\W)$ and $f\in C(\W)$, where the underlying viscosity solutions satisfy the PDE in a much weaker sense \cite{Crandall_Ishii_Lions}.  

Due to the lack of a divergence structure, constructing convergent numerical methods for \eqref{eqn1.1}, especially finite element methods, is not obvious.  
Only a handful of Galerkin-type methods have been developed, and these methods did not appear in the literature until quite recently \cite{Feng2017SS,AX:WW,AX:FN,SmearsSuli,NochettoZhang16}.  
All of these methods, however diverse they are in their construction, share a common thread: a (nonstandard) direct discretization of \eqref{eqn1.1}.  The method we propose, however, is based upon the Vanishing Moment Method (VMM) - a method developed by Feng and Neilan in \cite{Feng2009} for second-order fully nonlinear PDEs such as the Hamilton-Jacobi-Bellman and Monge-Amp\`ere equations.  The main solution concept for these equations is that of the viscosity solution which requires passing the derivatives of the solutions to functions that locally lie above or below the graph of the solution \cite{Crandall_Ishii_Lions}.  This notion of a solution is not natural in the Galerkin framework since it is not based on integration by parts.  The VMM seeks to approximate \eqref{eqn1.1} by a fourth order, quasi-linear PDE where the fourth order term is a ``nice'' operator, such as the biharmonic operator.  Since this new PDE is in divergence form if the biharmonic operator is chosen, it can be readily adapted to a weak solution concept and, more importantly, allow the natural formulation of Galerkin-type methods.    In our case of non-divergence form PDEs, the VMM is given by
\begin{align} \label{eqn1.4a}   \tag{$P_{\e}$}
\begin{split}  
\e\Delta^2u^\e - A:D^2u^\e+b\cdot\grad u^\e + cu^\e &= f \qquad \text{ in } \W, \\ 
u^\e &= 0 \qquad \text{ on } \partial\W,  \\
\Delta u^\e &= 0 \qquad \text{ on } \partial\W.
\end{split}
\end{align} \addtocounter{equation}{1}
While strong solution theory for \eqref{eqn1.1} is not as weak as viscosity solution theory for fully nonlinear PDEs, by first applying the VMM to \eqref{eqn1.1}, the 
resulting  approximate equation \eqref{eqn1.4a}, whose highest order derivative is in divergence form, can be discretized using a variety of conforming and nonconforming finite element methods.  These numerical solutions will converge to the solution of \eqref{eqn1.1} as $\e\to 0$.  Thus, from a numerical standpoint, the application of the VMM is just as applicable to non-divergence PDEs as they are to fully nonlinear PDEs.  

Several papers have be written on ways to formally construct solutions to fully nonlinear PDEs using the Vanishing Moment Method \cite{Feng2009,Feng2009a}.  This paper is the first to offer a detailed analysis of the VMM for a particular class of PDEs.  Moreover, since non-divergence PDEs have close ties to many popular second-order fully nonlinear PDEs, they serve as a natural starting point for the complete analysis of the Vanishing Moment Method.  

The goals of this paper are to provide a detailed PDE analysis of the Vanishing Moment Method when it is applied to second order elliptic linear PDEs in non-divergence form
and to study its finite element approximations.  This analysis requires showing that the solution $u^\e$ of \eqref{eqn1.4a} exists and is unique; proving $u^\e\to u$ as $\e\to 0$, where $u\in H^2(\W)\cap H_0^1(\W)$ is the strong solution to \eqref{eqn1.1}; and formulating error estimates for $\|u^\e-u\|$ in powers of $\e$.   
We also formulate a $C^1$ conforming finite element method for approximating the weak solution to \eqref{eqn1.4a} and derive its error estimate in the energy-norm as 
well as provide some numerical experiments that test the method and theory.  We note that the PDE results given for the VMM will be crucial 
for the analysis of any Galerkin-type scheme developed for \eqref{eqn1.4a}, including our $C^1$ scheme.  

The rest of our paper is organized as follows.  In Section 2, we set the notation and provide the preliminary information about the well-posedness and stability of \eqref{eqn1.1}.  In Section 3, we formally introduce and analyze the VMM applied to \eqref{eqn1.1}.  
In Section 4, we propose a simple $C^1$ conforming finite element method for our fourth order equation as well as give some tests showing the convergence of the method.  

\section{Notation and Preliminaries}

Let $\W\subset \R^n$ be an open and bounded domain.  Consider a subdomain $D\subset \W$.  Let $L^2(D)$ and $H^k(D)$ be the standard Lebesgue and Sobolev spaces with their respective norms.  We define $(\cdot,\cdot)_D$ to be the $L^2$-inner product for all scalar and vector valued functions with $(\cdot,\cdot):=(\cdot,\cdot)_\W$.  In addition, we let $H_0^k(\W)$ be the completion of all compactly supported smooth functions in $H^k(D)$ and $H^{-1}(D)$ be the dual space of $H_0^1(D)$.  
Lastly, define the space $H$ by 
\[
	H:=\{v\in H^2(\W)\cap H_0^1(\W) \text{ with } \Delta v\in H_0^1(\W)\}.
\] 

As mentioned in the introduction, there are three main well-posedness theories for \eqref{eqn1.1}.  We are choosing to focus on approximating 
$W^{2,p}$ strong solutions for linear elliptic PDEs.  To this end,
we assume that $A\in \left[C(\overline{\W})\right]^{n\times n}$ is uniformly elliptic; that is, there exist constants $0<\lambda\leq \overline{\lambda}$ such that
\begin{align}\label{eqn1.2}
\lambda |\xi|^2 \leq A(x)\xi\cdot\xi \leq \overline{\lambda}|\xi|^2\quad\quad \forall \xi\in\R^n,x\in\overline{\W} 
\end{align} 
 and $c\geq 0$ a.e.\ in $\W$.  Let $f\in L^p(\W)$. From \cite{Sp:DT} we have that \eqref{eqn1.1} exhibits a unique strong solution 
 $u\in W^{2,p}(\W)\cap W^{1,p}_0(\W)$  for $1<p<\infty$ that satisfies the PDE almost everywhere.  Moreover, we have the stability result
\begin{align} \label{eqn1.3}
\|u\|_{W^{2,p}(\W)} \lss \| \mL u \|_{L^p(\W)} .
\end{align}
Here and in the rest of the paper we use $a\lss b$ to denote $a\leq Cb$ for some constant $C>0$ independent of 
relevant parameters.
Estimate \eqref{eqn1.3} is called the Colder\'{o}n-Zygmond estimate for $\mL$.   
For simplicity of presentation, we will assume for the remainder of the paper that $b,c \equiv 0$.  

For the finite element method developed in Section 4, given an $h>0$, we let $\Th$ be a quasi-uniform and shape-regular mesh of $\W$.  We set $V_h\subset H^2(\W)\cap H_0^1(\W)$ to be a $C^1$ conforming finite element space over $\Th$ satisfying
\[
 V_h = \left\{v_h\in C^1(\overline\W)\cap H_0^1(\W) : v_h\big|_{T} \in \mathbb{P}_k(T)\ \forall T\in\Th \right\} , 
\]
where $\mathbb{P}_k(T)$ is the set of all polynomials of total degree less than or equal to $k$.  Examples of such spaces include those defined 
using the cubic Hermite element in $1$-D where, $k\geq 3$, or the Argyris element in $2$-D, where $k\geq 5$ \cite{Sp:BS}.  Given a sub-domain $D\subset\W$, we define
\[
 V_h = \left\{v_h\in V_h : v_h = 0 \text{ on }\W\setminus D \right\} 
\]
and remark that $V_h(D)$ is non-trivial provided $h\leq \frac{1}{3}\mbox{diam}(D)$.  

Define the discrete $L^2$ norm by
\begin{align}\label{discretedualL2}
	\| v \|_{L_h^2(D)} : = \sup_{w_h \in V_h(D)\setminus\{0\}}\frac{(v,w_h)_D}{\|w_h\|_{L^2(D)}}.
\end{align}
Clearly we have
\[
 \| v_h \|_{L^2(D)} \leq \| v_h \|_{L_h^2(D)} \quad\forall v_h\in V_h .  
\]
There also holds 
\[
 \| v \|_{L_h^2(D)} \lss \| v \|_{L^2(D)} \quad\forall v\in L^2(D) . 
\]
Lastly, define the discrete $H^{-2}$ norm by
\begin{align}\label{discretedualH2}
	\|v\|_{H_h^{-2}(D)} = \sup_{w_h \in V_h(D)\setminus\{0\}}\frac{(v,w_h)}{\|w_h\|_{H^2(D)}}.
\end{align}
For notational convenience, we will often write $w_h\in V_h(D)$ instead of $w_h\in V_h(D)\setminus\{0\}$ in \eqref{discretedualL2} and \eqref{discretedualH2}.

\section{The Vanishing Moment Method and its Analysis}

\subsection{ Construction of the Vanishing Moment Method}

The Vanishing Moment Method is an approximation technique originally developed for second-order fully nonlinear PDEs \cite{Feng2009}.  
The approximation corresponds to converting the original second-order equation into a quasi-linear fourth order equation.  Let $\e>0$.  
In the context of our problem, the solution $u$ to \eqref{eqn1.1} 
will be approximated by the solution $u^\e$, where $u^\e$ satisfies the fourth order problem:
\begin{align} \label{eqn1.4}  \tag{$P_{\e}$}
\begin{split} 
\mLe u^\e := \e\Delta^2u^\e - A:D^2u^\e &= f \qquad\text{ in } \W, \\ 
u^\e &= 0 \qquad \text{ on } \partial\W,  \\
\Delta u^\e &= 0 \qquad \text{ on } \partial\W.
\end{split}
\end{align} \addtocounter{equation}{1}

Since \eqref{eqn1.4} is a fourth order equation, an additional boundary condition must be added in order to guarantee a unique solution.  
We add the simply supported boundary condition $\Delta u^\e=0$ for this particular method, due to being a natural boundary condition for the biharmonic equation; moreover, it allows us to achieve $H^2$ estimates near $\partial\W$.  We refer to \cite{Feng2009} for other 
boundary conditions that may be used.  

Using the fact that the highest order derivative of \eqref{eqn1.4} is in divergence form, we can easily define a weak solution concept for \eqref{eqn1.4} as follows.

\begin{defn}
A function  $u^\e\in H^2(\W)\cap H_0^1(\W)$ is called a weak solution to \eqref{eqn1.4} if it satisfies 
\begin{align} \label{eqn:weaksoln}
\e \bigl( \Delta u^\e, \Delta v\bigr)- \bigl( A:D^2u^\e, v\bigr) = (f, v)  \qquad\forall v\in H^2(\W)\cap H_0^1(\W).
\end{align}
\end{defn}
Note that the simply supported boundary condition is naturally absorbed into the weak formulation.  
 
\subsection{Stability Estimates for $\mLe$ } \label{subsection3.1}

We now give the analysis required to show an estimate similar to the Calderon-Zygmund estimate \eqref{eqn1.3} but for the  operator $\mLe$ by using a freezing coefficient technique.  
This technique, used in \cite[Chapter 9]{Sp:DT} by Gilbarg-Trudinger to show \eqref{eqn1.3}, relies on that fact that since $A$ is continuous, it is approximately 
constant over small balls.  For $A$ constant, the non-divergence operator $-A:D^2u$ is merely a change of basis and dilation away from the Laplacian $-\Delta u$.  
Since \eqref{eqn1.3} is true for $\mL=-\Delta$, one can argue the estimate holds for all $A$ constant and holds locally for all $A$ continuous over a small ball.  
Then, using a partition of unity argument and cut-off functions, the estimate can be shown to hold over all of $\W$.  Section \ref{section3.1.2} shows the global 
Calderon-Zygmund estimates for $\mLe$ for constant coefficient $A$ while Section \ref{section3.2.2} shows the analogous estimates for continuous coefficient $A$.  

\subsubsection{Estimates for Constant Coefficient Operators} \label{section3.1.2}

First we consider the case $A=A_0$ for $A_0$ a constant matrix satisfying \eqref{eqn1.2}.  This gives us the following two problems: 

\begin{align} \label{eqn2.1} \tag{$P'$}
\begin{split} 
\mLz u:= - A_0:D^2u &= f \qquad  \text{ in } \W, \\ 
u &= 0 \qquad \text{ on } \partial\W
\end{split}
\end{align} \addtocounter{equation}{1}
and
\begin{align}  \label{eqn2.2} \tag{$P_{\e}'$}
\begin{split} 
\mLze u^\e:= \e\Delta^2u^\e - A_0:D^2u^\e &= f \qquad \text{ in } \W, \\ 
u^\e &= 0 \qquad\text{ on } \partial\W, \\
\Delta u^\e &= 0 \qquad \text{ on } \partial\W.
\end{split}
\end{align} \addtocounter{equation}{1}

It should be noted that since $A_0:D^2u = \div(A_0\grad u)$, we immediately recover weak solutions $u\in H_0^1(\W)$ and $u^\e\in H^2(\W)\cap H_0^1(\W)$ to \eqref{eqn2.1} and \eqref{eqn2.2}, respectively.  We seek to derive $H^1$ estimates for $u^\e$.

\begin{lemma} \label{lem2.1}
For $\e>0$, let $u^\e\in H^2(\W)\cap H_0^1(\W)$ be the weak solution of \eqref{eqn2.2}.  Then there holds the following estimate:
\begin{align}\label{eqn2.3}
\sqrt{\e}\|\Delta u^\e \|_{L^2(\W)} + \sqrt{\lambda}\| \grad u^\e \|_{L^2(\W)}  \lss \| f \|_{H^{-1}(\W)}.
\end{align}
Moreover, $u^\e\rightharpoonup u$ weakly in $H^1(\W)$, where $u$ is the weak solution to \eqref{eqn2.1}. 
\end{lemma}

\begin{proof}
Testing \eqref{eqn2.2} by $u^\e$, using the ellipticity condition \eqref{eqn1.2}, and applying 
integration by parts, we have 
\begin{align*}
\e\|\Delta u^\e \|_{L^2(\W)}^2 + \lambda\| \grad u^\e \|_{L^2(\W)}^2 &\leq \e(\Delta u^\e,\Delta u^\e) + (A_0\grad u^\e,\grad u^\e) \\ 
&= (f,u^\e) \\
&\leq \|f \|_{H^{-1}(\W)}\|u^\e \|_{H^{1}(\W)} \\
&\lss \|f \|_{H^{-1}(\W)}\|\grad u^\e \|_{L^2(\W)} \\
&\leq \frac{1}{2\d}\|f \|_{H^{-1}(\W)}^2 + \frac{\d}2 \|\grad u^\e \|_{L^2(\W)}^2 
\end{align*}
for any $\delta > 0$.  
Choosing $\d$, independent of $\e$, sufficiently small allows us to move $\|\grad u^\e \|_{L^2(\W)}$ on the right side 
to the left and obtain \eqref{eqn2.3}.  This estimate immediately gives us the boundedness of $\{u^\e\}$ in $H_0^1(\W)$.  
Thus, by the weak compactness of $H_0^1(\W)$, there exists a subsequence  $\{u^\e\}$ (not relabeled) and $u^*\in H_0^1(\W)$ such that 
$u^\e\rightharpoonup u^*$ weakly in $H^1(\W)$ as $\e \to 0$.  
Since $u^\e$ is a weak solution of \eqref{eqn2.2}, we have for every $\phi\in C_0^{\infty}(\W)$
\begin{align}\label{eqn2.3.1}
\e(\Delta u^\e,\Delta\phi) + (A_0\grad u^\e,\grad\phi) = (f,\phi) . 
\end{align}
Using the weak convergence of $u^\e\rightharpoonup u^*$, 
the uniform boundedness of $\e\|\Delta u^\e \|_{L^2(\W)}^2$, 
and passing $\e\to 0$ in \eqref{eqn2.3.1}, we obtain
\[
	(A_0\grad u^*,\grad\phi) = (f,\phi).
\]
Thus, $u^*$ is a weak solution of \eqref{eqn2.1}.  By uniqueness we have $u^* = u$, and the whole sequence 
$u^\e\rightharpoonup u$ weakly in $H^1(\W)$.  The proof is complete.
\end{proof}

We now wish to derive local $H^2$ stability estimates for $\mLe$.  

\begin{lemma} \label{lem2.2}
Let $B\subset\subset\W$ and $v\in H$ with $\supp(v)\subset B$.  
The following estimate holds: 
\begin{align} \label{eqn2.4}
\sqrt{\e}\|\grad\Delta v \|_{L^2(B)} + \sqrt{\lambda}\| D^2 v \|_{L^2(B)} \lss \|\mLze v \|_{L^2(B)} . 
\end{align}
\end{lemma}

\begin{proof}
By testing $\mLze v$ by $-\Delta v$ and integrating by parts, we get
\begin{align}\label{eqn2.5}
(\mLze v,-\Delta v) = (\e \Delta^2 v-A_0:D^2v,-\Delta v) = \e\|\grad\Delta v \|_{L^2(B)}^2 + (A_0:D^2v,\Delta v).
\end{align}
Since $A_0$ is symmetric and positive definite, there exists an orthogonal matrix $Q\in\R^{n\times n}$ such that 
$Q^TA_0Q = \diag(\lambda_1,\lambda_2,\ldots,\lambda_n) =: \Lambda$, where $\lambda\leq \lambda_1\leq \lambda_2 \leq \cdots\leq \lambda_n$.  
Let $y=Q^T x$ and $\hat{v}(y) = v(Qy) = v(x)$.  Since the Laplacian is preserved under an orthogonal change of basis, we have the following:
\begin{align*}
\Delta_x v(x) &= \Delta_y \hat{v}(y), & \Delta_x^2 v(x) &= \Delta_y^2 \hat{v}(y),
\end{align*}
\begin{align*}
A_0:D_x^2 v(x) = \Lambda:D_y^2\hat{v}(y) = \sum_{j=1}^n\lambda_j \hat{v}_{y_j y_j}(y).
\end{align*}
Without a loss of generality, we may assume that $A_0 = \Lambda$ in \eqref{eqn2.5}.  Hence, 
\begin{align} \label{eqn2.6}
\begin{split}
(A_0:D^2v,\Delta v) &= (\div(A_0\grad v),\Delta v)  \\
&= -(A_0\grad v,\grad\Delta v) \\
&= -(A_0\grad v,\div(D^2v)) \\
&= (\grad(A_0\grad v),D^2v) \\
&= \sum_{j=1}^n\lambda_j \| \grad v_{x_j} \|_{L^2(B)}^2 \geq \lambda\| D^2 v \|_{L^2(B)}^2.
\end{split}
\end{align}
Combining \eqref{eqn2.5} and \eqref{eqn2.6} gives us
\begin{align*}
\e\|\grad\Delta v \|_{L^2(B)}^2 + \lambda \| D^2 v \|_{L^2(B)}^2 &\leq \| \mLze v \|_{L^2(B)}\| \Delta v \|_{L^2(B)} \\
&\leq \frac{\delta}{2} \| \Delta v \|_{L^2(B)}^2 + \frac{1}{2\d}  \| \mLze v \|_{L^2(B)}^2 \\
&\leq \frac{\delta}{2} \| D^2v \|_{L^2(B)}^2 + \frac{1}{2\d}  \| \mLze v \|_{L^2(B)}^2 .
\end{align*}
Choosing $\d$ sufficiently small to move $\| D^2 v \|_{L^2(B)}^2$ to the left hand side gives the desired result.  The proof is complete.
\end{proof}

Next, we derive a similar boundary estimate.  Let $\R_+^n:=\{x=(x',x_n)\in\R^n:x_n>0\}$, $B^+=B\cap \R_+^n$, 
and $(\partial B^+)^+ = \partial(B^+)\cap \R_+^d$, where $B$ is a small ball with its center on the $x_n$-axis.  

\begin{lemma} \label{lem2.1b}
Let $v\in H^2(B)\cap H_0^1(B)$ with $\Delta v\in H_0^1(B)$ and $v,\grad v,\Delta v = 0$ near $\partial B^+$.  Then we have the following estimate:
\begin{align} \label{eqn2.1b}
\sqrt{\e}\|\grad\Delta v \|_{L^2(B^+)} + \lambda\| D^2v \|_{L^2(B^+)} \lss \|\mLze v \|_{L^2(B^+)} . 
\end{align}
\end{lemma}

\begin{proof}
We will extend $v$ from $B^+$ to $B$ by an odd reflection, that is $v(x',x_n) = -v(x',-x_n)$ for all $x\in B\setminus B^+$. Since $\supp(v)\subset B$ 
after the reflection, we may test the PDE by $\Delta v$ and use a similar argument as to the one in Lemma \ref{lem2.2} to obtain 
\begin{align} \label{eqn2.2b}
\sqrt{\e}\|\grad\Delta v \|_{L^2(B)} + \sqrt{\lambda}\| D^2 v \|_{L^2(B)} \lss \|\mLze v \|_{L^2(B)} . 
\end{align}
Since the odd reflection is a bounded linear operator, we have
\begin{align*}
\sqrt{\e}\|\grad\Delta v \|_{L^2(B^+)} + \sqrt{\lambda}\| D^2 v \|_{L^2(B^+)} &\leq \sqrt{\e}\|\grad\Delta v \|_{L^2(B)} + \sqrt{\lambda}\| D^2 v \|_{L^2(B)}  \\
&\lss\|\mLze v \|_{L^2(B)} \\
&\lss\|\mLze v \|_{L^2(B^+)} . 
\end{align*} 
The proof is complete.
\end{proof}

\subsubsection{Estimates for Continuous Coefficient Operators} \label{section3.2.2}

Suppose that  $A\in \left[C(\overline{\W})\right]^{n\times n}$ is uniformly positive definite.  In this section, we seek uniform $H^1$ and $H^2$ stability 
estimates for $\mLe$. Following the freezing coefficients technique, we first need to derive local $H^1$ and $H^2$ stability estimates, which in 
turn require the following lemma controlling the bound of the $H^{-1}$ norm of the Hessian.

\begin{lemma} \label{hm1lemma}
Let $B$ be an open ball and $v\in H^2(B)$ and $|\cdot|_1$ denote the vector $1$-norm.  There holds 
\begin{align}\label{h1control}
\| D^2 v\|_{H^{-1}(B)} &\leq n^{\frac12} \| \grad v \|_{L^2(B)}, \\
\| |D^2 v|_1\|_{H^{-1}(B)} &\leq n\| \grad v \|_{L^2(B)}, \label{h1control_2}
\end{align}
where $n$ is the dimension of the domain $\W$ and 
\begin{align*} 
\| D^2 v\|_{H^{-1}(B)} &= \left(\sum_{i,j=1}^n \| v_{x_ix_j} \|_{H^{-1}(B)}^2\right)^{\frac12},\\
\| |D^2 v|_1 \|_{H^{-1}(B)} &=  \sum_{i,j=1}^n \bigl\| |v_{x_ix_j}| \bigr\|_{H^{-1}(B)} .
\end{align*}
\end{lemma}

\begin{proof}
Let $i,j=1,\ldots,n$ and $w\in H_0^1(B)$ with $w\not\equiv 0$.  Integration by parts yields
\begin{align*}
( v_{x_i x_j} , w ) = -(v_{x_i},w_{x_j}) \leq \|v_{x_i}\|_{L^2(B)}\|\grad w\|_{L^2(B)}.  
\end{align*}
Thus, by the definition of $\|v\|_{H^{-1}(B)}$, we have
\begin{align*}
\|v_{x_i x_j}\|_{H^{-1}(B)} = \sup_{\substack{w\in H_0^1(B) \\ w\not\equiv 0}}\frac{(v_{x_i x_j},w)_B}{\|\grad w\|_{L^2(B)}} \leq \|v_{x_i}\|_{L^2(B)}.
\end{align*}
Summing over $i$ and $j$ gives us \eqref{h1control}.  

The proof of \eqref{h1control_2} follows the same lines with an additional help of some facts from \cite[Chapter 7]{Sp:DT}.  Let 
$|v_{x_i,x_j}|= v^+_{x_ix_j} + v^-_{x_ix_j}$, 
 where $v^\pm$ denotes perspectively the positive and 
negative parts of $v$.  Then
\begin{align*}
 (v_{x_ix_j}^{+},w) &= \int_{\{v_{x_i,x_j}>0\}}v_{x_ix_j}(x)w(x) \dx[x] = -\int_{\{v_{x_i,x_j}>0\}} v_{x_i}(x)w_{x_j}(x)\dx[x] \\ &\leq \|v_{x_i}\|_{L^2(\W)}\|\grad w\|_{L^2(\W)}.
\end{align*} 
A similar result holds for $v_{x_ix_j}^-$.  The proof is complete.
\end{proof}

We are now ready to prove the local $H^1$ and $H^2$ stability of $\mLe$.

\begin{lemma} \label{lem3.1}
Let $x_0\in\W$ and $B_R(x_0)\subset\W$ be the ball of radius $R$ centered at $x_0$. There exists $R_{\d}>0$, independent 
of $\e$, such that, for all $v\in H$ with $\supp(v)\subset B:= B_{R_{\d}}(x_0)$, the following estimates hold: 
\begin{align} \label{eqn3.1}
\sqrt{\e}\|\grad\Delta v \|_{L^2(B)} + \sqrt{\lambda}\| D^2 v \|_{L^2(B)} &\lss \|\mLe v \|_{L^2(B)}, \\
\sqrt{\e}\|\Delta v \|_{L^2(B)} + \sqrt{\lambda}\| \grad v \|_{L^2(B)} &\lss \|\mLe v \|_{H^{-1}(B)}. \label{eqn3.1a}
\end{align}
\end{lemma}

\begin{proof}
Let $\d>0$, and define $A_0:=A(x_0)$.  Since $A$ is continuous, there exists $R_{\d}>0$ such that 
\[
	\| A - A_0 \|_{L^{\infty}(B_{R_{\d}}(x_0))} \leq \d.
\]
By Lemma \ref{lem2.2}, we have
\begin{align*}
\begin{split}
\sqrt{\e}\|\grad\Delta v \|_{L^2(B)} + \sqrt{\lambda}\| D^2 v \|_{L^2(B)} &\lss \|\mLze v \|_{L^2(B)} \\
&\lss \| \mLe v \|_{L^2(B)} + \|(\mLze - \mLe) v \|_{L^2(B)} \\
&= \| \mLe v \|_{L^2(B)} + \| (A - A_0):D^2 v \|_{L^2(B)}  \\
&\lss \| \mLe v \|_{L^2(B)} + \| A - A_0\|_{L^{\infty}(B)} \| D^2 v \|_{L^2(B)} \\
&\leq \| \mLe v \|_{L^2(B)} +  \d \| D^2 v \|_{L^2(B)}, 
\end{split}
\end{align*}
and \eqref{eqn3.1} follows for $\d$ sufficiently small (independent of $\e$ and $B$).  

To show \eqref{eqn3.1a}, we follow a similar technique.  Using Lemma \ref{lem2.1} and Lemma \ref{hm1lemma}, we have
\begin{align*}
\begin{split}
\sqrt{\e}\|\Delta v \|_{L^2(B)} + \sqrt{\lambda}\| \grad v \|_{L^2(B)} &\lss \|\mLze v \|_{H^{-1}(B)} \\
&\lss \| \mLe v \|_{H^{-1}(B)} + \|(\mLze - \mLe) v \|_{H^{-1}(B)} \\
&= \| \mLe v \|_{H^{-1}(B)} + \| (A - A_0):D^2 v \|_{H^{-1}(B)}  \\
&\lss \| \mLe v \|_{H^{-1}(B)} + \| A - A_0\|_{L^{\infty}(B)} \| |D^2 v|_1 \|_{H^{-1}(B)} \\
&\lss \| \mLe v \|_{H^{-1}(B)} +  \d n^2 \| \grad v \|_{L^2(B)}, 
\end{split}
\end{align*}
and \eqref{eqn3.1a} follows for $\d$ sufficiently small (independent of $\e$ and $B$).  
The proof is complete.
\end{proof}

Finally, using a partition of unity argument, we can get  an interior G\r{a}rding inequality.  

\begin{lemma} \label{lem3.2}
Let $v\in H$. For any $\W'\subset\subset \W$ the following estimates hold: 
\begin{align} \label{eqn3.2}
\sqrt{\e}\|\grad\Delta v \|_{L^2(\W')} &+ \sqrt{\lambda}\| D^2 v \|_{L^2(\W')} \lss \|\mLe v \|_{L^2(\W)} + \| v \|_{L^2(\W)}  \\
&\hskip 0.8in + \e(\| \grad v \|_{L^2(\W)} +  \| \Delta v \|_{L^2(\W)} + \| \grad\Delta v \|_{L^2(\W)}),\nonumber \\
\sqrt{\e}\|\Delta v \|_{L^2(\W')} &+ \sqrt{\lambda}\| \grad v \|_{L^2(\W')}\lss \|\mLe v \|_{H^{-1}(\W)} + \| v \|_{L^2(\W)} \nonumber\\
&\qquad +   \e(\| \grad v \|_{L^2(\W)} +  \| \Delta v \|_{L^2(\W)}) .  \label{eqn3.2a}
\end{align}
\end{lemma}

\begin{proof}
For a ball $B_R$ with radius $R$, let $\sigma = 1/2$ and consider the cutoff function $\eta\in C_0^{\infty}(B_R)$ with $0\leq\eta\leq 1$, $\eta \equiv 1$ in $B_{\sigma R}$, and $\eta \equiv 0$ on $B_R\setminus B_{\sigma'R}$, where $\sigma' = 3/4$.  Moreover, $\| D^k \eta \|_{L^\infty(B_R)} \lss (1-\sigma)^{-k}R^{-k}$ for $k=0,1,2,3,4$.  Applying \eqref{eqn3.1} to the function $ \eta v$ on the ball $B_{\sigma R}$ gives us
\begin{align} \label{eqn3.3}
\begin{split}
\sqrt{\e}\|\grad\Delta v \|_{L^2(B_{\sigma R})} &+ \sqrt{\lambda}\| D^2 v \|_{L^2(B_{\sigma R})} \\
&= \sqrt{\e}\|\grad\Delta(\eta v) \|_{L^2(B_{\sigma R})} + \sqrt{\lambda}\| D^2 (\eta v) \|_{L^2(B_{\sigma R})} \\
&\leq \sqrt{\e}\|\grad\Delta(\eta v) \|_{L^2(B_{\sigma'R})} + \sqrt{\lambda}\| D^2 (\eta v) \|_{L^2(B_{\sigma'R})} \\
&\lss  \|\mLe (\eta v) \|_{L^2(B_{\sigma'R})} = \| \e\Delta^2(\eta v) - A:D^2(\eta v) \|_{L^2(B_{\sigma'R})} . 
\end{split}
\end{align}
Expanding $\Delta^2(\eta v)$ and $A:D^2(\eta v)$ gives us
\begin{align} \label{eqn3.4}
\Delta^2(\eta v) &= \eta\Delta^2 v + 4 \grad\Delta v\cdot\grad\eta + 6\Delta v\Delta \eta  + 4\grad v\cdot\grad\Delta\eta + v\Delta^2\eta, \\ \label{eqn3.5}
A:D^2(\eta v) &=\eta A:D^2v + 2A\grad v\cdot\grad\eta + v A:D^2\eta.
\end{align}
Using \eqref{eqn3.4} and \eqref{eqn3.5} gives us, with the $L^2$ norm taken over $B_{\sigma'R}$,
\begin{align} \label{eqn3.6}
\| \e\Delta^2(\eta v) - A:D^2(\eta v) \|_{L^2} &\lss \| \mLe v \|_{L^2} + \frac{1}{(1-\sigma)R}\|\grad v \|_{L^2} + \frac{1}{(1-\sigma)^2R^2}\|v  \|_{L^2} \\
&\quad+ \frac{\e}{(1-\sigma)^4R^4}\| v \|_{L^2}  + \frac{\e}{(1-\sigma)^3R^3}\| \grad v \|_{L^2}
\nonumber \\
&\quad+ \frac{\e}{(1-\sigma)^2R^2}\| \Delta v \|_{L^2} + \frac{\e}{(1-\sigma)R}\| \grad\Delta v \|_{L^2} . 
\nonumber
\end{align}
The treatment of the first three terms on the right follows from the interpolation technique in \cite[p.236]{Sp:DT}.  
Keeping the rest of the terms on the right and using a covering argument we arrive at \eqref{eqn3.2}. Since $\overline{\W'}$ 
is compact it will only take a finite number of balls to cover $\W'$. Thus, the estimate does not depend on $R$. 

To show \eqref{eqn3.2a}, using the same $\eta$ as prescribed above we first recover a similar estimate to \eqref{eqn3.3} for \eqref{eqn3.2a}: 
\begin{align} \label{eqn3.16}
\begin{split}
\sqrt{\e}\|\Delta v \|_{L^2(B_{\sigma R})} +& \sqrt{\lambda}\| \grad v \|_{L^2(B_{\sigma R})} \\
&\lss  \|\mLe (\eta v) \|_{H^{-1}(B_{\sigma'R})} \\
&= \| \e\Delta^2(\eta v) - A:D^2(\eta v) \|_{H^{-1}(B_{\sigma'R})} \\
&= \sup_{\substack{w\in H_0^1(B_{\sigma'R}) \\ w\not\equiv 0}}\frac{(\e\Delta^2(\eta v) - A:D^2(\eta v),w)}{\|\grad w \|_{L^2(B_{\sigma' R})}}.
\end{split}
\end{align}
Let $w\in H_0^1(B_{\sigma'R}))$. By integration by parts we have
\begin{align} \label{eqn3.17}
\begin{split}
(\e\Delta^2(\eta v) - A:D^2(\eta v),w) &= -\e(\grad\Delta(\eta v),\grad w) - (A:D^2(\eta v),w) \\
&:=\e I_1 + I_2 . 
\end{split}
\end{align}
We first focus on $I_2$.  Expanding $\grad\Delta(\eta v)$ similar to \eqref{eqn3.4} and integrating by parts gives us
\begin{align*} \stepcounter{equation}\tag{\theequation}\label{eqn3.18}
I_1 &= -(\grad\Delta(\eta v),\grad w) \\
&= -(\eta\grad\Delta v,\grad w) - (3\Delta v\grad\eta + 3\Delta\eta\grad v + u^\e\grad\Delta\eta,\grad w) \\
&= (\div(\eta\grad\Delta v),w) - (3\Delta v\grad\eta + 3\Delta\eta\grad v + u^\e\grad\Delta\eta,\grad w)\\
&= (\Delta^2 v,\eta w) + (\grad\Delta v,\grad\eta w)  - (3\Delta v\grad\eta + 3\Delta\eta\grad v + u^\e\grad\Delta\eta,\grad w)\\
&= (\Delta^2 v,\eta w) - (\Delta v,\div(\grad\eta w))  - (3\Delta v\grad\eta + 3\Delta\eta\grad v + u^\e\grad\Delta\eta,\grad w)\\
&= (\Delta^2 v,\eta w) - (\Delta v,w\Delta\eta + \grad\eta\cdot\grad w)  - (3\Delta v\grad\eta + 3\Delta\eta\grad v + v\grad\Delta\eta,\grad w)\\
&\lss (\Delta^2 v,\eta w) + \left( \frac{1}{(1-\sigma)^2R^2}\| \Delta v\|_{L^2}  + \frac{1}{(1-\sigma)^2R^2}\| \grad v\|_{L^2} \right. \\
&\quad\quad \left .+ \frac{1}{(1-\sigma)^3R^3}\| v\|_{L^2}\right)\|\grad w \|_{L^2} . 
\end{align*}
Using \eqref{eqn3.5} on $I_2$, we get
\begin{align*} \stepcounter{equation}\tag{\theequation}\label{eqn3.19}
I_2 &= -(A:D^2(\eta v),w) \\
&= -(\eta A:D^2v + 2A\grad v\cdot\grad\eta + v A:D^2\eta,w) \\
&\lss -(A:D^2v,\eta w) + \left( \frac{1}{(1-\sigma)R}\| \grad v\|_{L^2} + \frac{1}{(1-\sigma)^2R^2}\| v\|_{L^2}\right)\|\grad w \|_{L^2} . 
\end{align*}
Combining \eqref{eqn3.18} and \eqref{eqn3.19} with \eqref{eqn3.17} gives us
\begin{align}
\begin{split}
\sqrt{\e}\|\Delta v \|_{L^2(B_{\sigma R})} +& \sqrt{\lambda}\| \grad v \|_{L^2(B_{\sigma R})} \\
&\lss \| \mLe v \|_{H^{-1}(\W) } + \frac{1}{(1-\sigma)R}\|\grad v \|_{L^2} + \frac{1}{(1-\sigma)^2R^2}\|v  \|_{L^2} \\
&\quad+ \frac{\e}{(1-\sigma)^3R^3}\| v \|_{L^2}  + \frac{\e}{(1-\sigma)^2R^2}\| \grad v \|_{L^2} \\
&\quad+ \frac{\e}{(1-\sigma)^2R^2}\| \Delta v \|_{L^2} . 
\end{split}
\end{align}
Following a similar treatment as \eqref{eqn3.6}, we arrive at \eqref{eqn3.2a}.  
The proof is complete.  
\end{proof}

We now desire a global estimate instead of an interior estimate.

\begin{lemma} \label{lem3.3}
Let $\partial\W\in C^{2,1}$ and $v\in H$.  The following estimates hold: 
\begin{align} \label{eqn3.7}
\sqrt{\e}\|\grad\Delta v \|_{L^2(\W)} + \sqrt{\lambda}\| D^2 v \|_{L^2(\W)} &\lss \|\mLe v \|_{L^2(\W)} + \| v \|_{L^2(\W)} + \e\big(\| \grad v \|_{L^2(\W)} \\
&\quad+  \| \Delta v \|_{L^2(\W)} + \| \grad\Delta v \|_{L^2(\W)}\big), \nonumber \\ \label{eqn3.7a}
\sqrt{\e}\|\Delta v \|_{L^2(\W)} + \sqrt{\lambda}\|\grad v \|_{L^2(\W)} &\lss \|\mLe v \|_{H^{-1}(\W)} 
+ \| v \|_{L^2(\W)} \\
&\quad+ \e(\| \grad v \|_{L^2(\W)} +  \| \Delta v \|_{L^2(\W)}).  \nonumber 
\end{align}
\end{lemma}

\begin{proof}
Since $\partial\W\in C^{2,1}$, for any $x_0\in\partial\W$, we may flatten $\partial\W$ near $x_0$ and use Lemma \ref{lem2.1b} and the proof in Lemma \ref{lem3.1} to create a local boundary estimate mimicking \eqref{eqn3.1} and \eqref{eqn3.1a}.  Following the same argument as in Lemma \ref{lem3.2} we can obtain estimates \eqref{eqn3.2} and \eqref{eqn3.2a} near the boundary.  These estimates combined with \eqref{eqn3.2} and \eqref{eqn3.2a} give us \eqref{eqn3.7} and \eqref{eqn3.7a}.  The proof is complete.
\end{proof}

Now we must deal with the terms involving $\e$ on the right hand side of \eqref{eqn3.7} and \eqref{eqn3.8}.  
Noting that $\e/\sqrt{\e}\to 0$ as $\e\to 0$, we may hide these terms for $\e$ sufficiently small.   

\begin{lemma} \label{lem3.4}
Let $\partial\W\in C^{2,1}$ and $v\in H$.  There exits $\e_0>0$ such that for any $\e<\e_0$ the following estimates hold: 
\begin{align} \label{eqn3.8}
&\sqrt{\e}\|\grad\Delta v \|_{L^2(\W)} + \sqrt{\lambda}\| D^2 v \|_{L^2(\W)} + \sqrt{\lambda}\| \grad v \|_{L^2(\W)} \lss \|\mLe v \|_{L^2(\W)} + \| v \|_{L^2(\W)} , \\
&\sqrt{\e}\|\Delta v \|_{L^2(\W)} + \sqrt{\lambda}\| \grad v \|_{L^2(\W)} \lss \|\mLe v \|_{H^{-1}(\W)} + \| v \|_{L^2(\W)} . \label{eqn3.8a}
\end{align}
\end{lemma}

\begin{proof}
Adding \eqref{eqn3.7} and \eqref{eqn3.7a} and noting that $\| \mLe v \|_{H^{-1}(\W)}\leq \| \mLe v \|_{L^2(\W)}$  yields 
\begin{align} \label{eqn3.9}
&\sqrt{\e}\|\grad\Delta v \|_{L^2(\W)} +\sqrt{\e}\|\Delta v \|_{L^2(\W)} + \sqrt{\lambda}\| D^2 v \|_{L^2(\W)} + \sqrt{\lambda}\| \grad v \|_{L^2(\W)} \\
&\, \leq C\left( \|\mLe v \|_{L^2(\W)} + \| v \|_{L^2(\W)} + \e(\| \grad v \|_{L^2(\W)} +  \| \Delta v \|_{L^2(\W)} + \| \grad\Delta v \|_{L^2(\W)}) \right) , \nonumber
\end{align}
where $C$ is independent of $\e$.  Choosing $\e_0 = \min\{4/C^2,\sqrt{\lambda}/(2C)\}$ 
gives us $C\e < \sqrt{\e}/2$ and $C\e  < \sqrt{\lambda}/2$ for all $\e<\e_0$.  
Letting $\e< \e_0$, we can move the terms with coefficient $\e$ on the right hand 
side of \eqref{eqn3.9} and obtain
\begin{align} \label{eqn3.10}
\begin{split}
\sqrt{\e}\|\grad\Delta v \|_{L^2(\W)} +\sqrt{\e}\|\Delta v \|_{L^2(\W)} +& \sqrt{\lambda}\| D^2 v \|_{L^2(\W)} + \sqrt{\lambda}\| \grad v \|_{L^2(\W)} \\
&\leq C\left( \|\mLe v \|_{L^2(\W)} + \| v \|_{L^2(\W)} \right).
\end{split}
\end{align}
Dropping $\sqrt{\e}\|\Delta v \|_{L^2(\W)}$ gives us \eqref{eqn3.8}.  
\eqref{eqn3.8a} is an immediate consequence of \eqref{eqn3.7a}.  The proof is complete.
\end{proof}

We can derive a full stability estimate from Lemma \ref{lem3.4} through contradiction and showing that $\mLe$ is 1-1.  

\begin{lemma} \label{lem3.5}
Let $\partial\W\in C^{2,1}$ and $v\in H$.  For all $\e < \e_0$, we have the following stability estimate:
\begin{align} \label{eqn3.11}
\sqrt{\e}\|\grad\Delta v \|_{L^2(\W)} + \sqrt{\lambda}\| D^2 v \|_{L^2(\W)} + \sqrt{\lambda}\| \grad v \|_{L^2(\W)} &\leq C\|\mLe v \|_{L^2(\W)} , \\ 
\sqrt{\e}\|\Delta v\|_{L^2(\W)} + \sqrt{\lambda}\| \grad v \|_{L^2(\W)} &\leq C\|\mLe v \|_{H^{-1}(\W)} , \label{eqn3.11a}
\end{align}
where $C$ in independent of $\e$ and $u^\e$.
\end{lemma}

\begin{proof}
Fix $\e< \e_0$.  
We first show that the operator $\mLe$ is 1-1 using eigenvalue theory.  
Define $H^{-2}(\W) = (H^2(\W)\cap H_0^1(\W))^*$ and $\mK:H^2(\W)\cap H_0^1(\W)\to H^{-2}(\W)$ by
\[
	(\mK w,v) = (\Delta w,\Delta v) \quad\forall v\in H^2(\W)\cap H_0^1(\W).
\]
By the elliptic existence theory for fourth order problems, $\mK$ is invertible with $\mK^{-1}$ bounded.
We see that the operator $\mLe=\e\mK + \mL$  satisfies
\begin{align}\label{eqn1.6}
(\mLe w,v) = \e(\Delta w,\Delta v) - (A:D^2 w,v)  \quad \forall v\in H^2(\W)\cap H_0^1(\W).  
\end{align}
Thus, a weak solution $u^\e$ to \eqref{eqn1.4} satisfies $\mLe u^\e = f$ in $H^{-2}(\W)$.     

Suppose $\mLe$ is not 1-1.  
Then there exists $w \neq 0$ such that $\mLe w = \mL w + \e\mK w = 0$ in $H^{-2}(\W)$.  Hence, $\mL w = -\mK(\e w)$ in $H^{-2}(\W)$.  Since $\mK$ is invertible we have, $\mKinv\mL w = -\e w$ in $H^2(\W)$. Thus, $-\e$ is an eigenvalue for $\mKinv\mL$.  Since $\mKinv$ is a symmetric and compact operator, its eigenvalues are positive and tend to 0.  Moreover, the eigenvalues of $\mL$, while complex, have real parts greater than some positive constant $r$ \cite[p.361]{Evans2010}.  
Since $ \mK^{-1}$ is positive definite, 
repeating the proof given in \cite[p.361]{Evans2010}, we can show that 
the eigenvalues of $\mKinv\mL$ must have positive real parts
(one may verify easily the conclusion in the finite-dimensional case).
Thus $-\e$ cannot be an eigenvalue, which is a contradiction.   Hence,  $\mLe$ must be 1-1.

To show \eqref{eqn3.11} we argue by contradiction.  Suppose that, for all $k\inN$, there exists $v_k\in H$ with $\|v_k\|_{L^2(\W)}=1$ and $\|\mLe v_k \|_{L^2(\W)}\to 0$ as $k\to\infty$.  By Lemma \ref{lem3.4} we have  $\sqrt{\e}\|\grad\Delta v_k\|_{L^2(\W)} + \| v_k \|_{H^2(\W)}$ is uniformly bounded in $k$.  Thus, we way may extract a convergent subsequence (not relabeled) and a $v_*\in H$ such that $v_k\wto v_*$ weakly in $H$ and $\sqrt{\e}\|\grad\Delta v_*\|_{L^2(\W)} + \| v_*\|_{H^2(\W)} > 0$.  Moreover, since $v_k\wto v_*$ and $\mLe$ is linear, we have $\mLe v_k\wto \mLe v_*$ in $H^{-2}(\W)$ and  
\[
0 \leq \|\mLe v_* \|_{H^{-2}(\W)} \leq \liminf_{k\to\infty} \|\mLe v_k \|_{H^{-2}(\W)} \leq \liminf_{k\to\infty} \|\mLe v_k \|_{L^2(\W)} = 0.
\]
Thus, $\|\mLe v_* \|_{H^{-2}(\W)}=0$ and consequently $\mLe v_* = 0$ in $H^{-2}(\W)$.  
Clearly $u^\e \equiv 0$ is a solution to \eqref{eqn1.4} when $f \equiv 0$.  
Since $\mLe$ is 1-1, 
then $v_*=0$ in $H$ which is a contradiction to the fact that $\sqrt{\e}\|\grad\Delta v_*\|_{L^2(\W)}  + \| v_*\|_{H^2(\W)} > 0$. The proof is complete.

\end{proof}

We are now ready to show the well-posedness of $\mLe$.

\begin{theorem} \label{thm:exist}
Let $\e<\e_0$ and $\partial\W\in C^{2,1}$.  Then, for every $f\in L^2(\W)$, there exists a unique weak solution $u^\e\in H$ to \eqref{eqn1.4}.  Moreover, the weak solution satisfies
\begin{align} \label{eqn3.11stab}
\sqrt{\e}\|\grad\Delta u^\e \|_{L^2(\W)} + \sqrt{\lambda}\| D^2 u^\e \|_{L^2(\W)} + \sqrt{\lambda}\| \grad u^\e \|_{L^2(\W)} &\leq C\|f \|_{L^2(\W)} , \\ 
\sqrt{\e}\|\Delta u^\e \|_{L^2(\W)} + \sqrt{\lambda}\| \grad u^\e \|_{L^2(\W)} &\leq C\| f\|_{H^{-1}(\W)} . \label{eqn3.11staba}
\end{align}
\end{theorem}  

\begin{proof}
We approximate $A$ by $A_k\in [C^1(\overline\W)]^{n\times n}$, where $A_k\to A$ uniformly in $\W$.  
Since $A_k$ is differentiable, 
\[
-A_k:D^2 v = -\div(A_k\grad v)+ \div(A_k)\cdot \grad v , 
\]
which converts the non-divergence operator into a sum of a diffusion operator and an advection operator.  Consider weak solutions to the problem 
\begin{align}  \label{eqnapprox}
\begin{split} 
 \e\Delta^2v -\div(A_k\grad v) + \div(A_k)\cdot \grad v &= f \text{ in } \W , \\ 
v &= 0 \text{ on } \partial\W , \\
\Delta v &= 0 \text{ on } \partial\W.
\end{split}
\end{align} \addtocounter{equation}{1}
By the fourth-order elliptic PDE theory \cite[Chapter 8]{Agmon10}, there is a weak solution $u_k^\e\in H^2(\W)\cap H_0^1(\W)$ to \eqref{eqnapprox}.  Moreover, $u_k^\e\in H$ since $\partial\W\in C^{2,1}$ and, 
by Lemma~\ref{lem3.5}, $u_k^\e$ satisfies
\begin{align} \label{eqn3.11.1}
\sqrt{\e}\|\grad\Delta u_k^\e \|_{L^2(\W)} + \sqrt{\lambda}\| D^2 u_k^\e \|_{L^2(\W)} + \sqrt{\lambda}\| \grad u_k^\e \|_{L^2(\W)} &\lss \|f \|_{L^2(\W)} , \\ 
\sqrt{\e}\|\Delta u_k^\e \|_{L^2(\W)} + \sqrt{\lambda}\| \grad u_k^\e \|_{L^2(\W)} &\lss \|f \|_{H^{-1}(\W)} . \label{eqn3.11.2}
\end{align}
Thus, $\{u_k^\e\}_k$ is bounded in $H$, and it follows that $u_k^\e$ weakly converges to some $u^\e\in H$.  By the linearity of the PDE, $u^\e$ is a weak solution of \eqref{eqn1.4}.  Since $\mLe$ is 1-1, the solution $u^\e$ is unique.  Taking the $\liminf$ as $k\to\infty$ of \eqref{eqn3.11.1} and \eqref{eqn3.11.2} we arrive at \eqref{eqn3.11stab} and \eqref{eqn3.11staba}.  The proof is complete.
\end{proof}

\begin{remark}
In the above three lemmas we require that $\partial \W\in C^{2,1}$. This is imposed in order to achieve the global estimates using a boundary flattening technique. More precisely, we use the assumption to preserve $\|\grad\Delta v\|_{L^2}$ when flattening the boundary, $\partial\W$.  Thus, to guarantee $u^\e\in H$, we need $\partial\W\in C^{2,1}$.  We note that if $\partial\W\in C^{1,1}$, the argument still 
works, but all $\|\grad\Delta v\|_{L^2}$ terms must be dropped implying $u^\e$ is only in 
$H^2$. This does not affect convergence but does affect the $H^1$ error estimate in Theorem \ref{thm_eps_rates} below.  
\end{remark}

We now prove the convergence of the solutions $u^\e\to u$, where $u$ is a strong solution to \eqref{eqn1.1}.  We also give an $H^1$ stability result which will be useful for the numerical discretization of \eqref{eqn1.1} and has not yet been obtained in the literature. 

\begin{theorem} 
Let $\e<\e_0$ and $u^\e$ be the solution to \eqref{eqn1.4}.  Then $u^\e$ converges to $u\in H^2(\W)\cap H_0^1(\W)$ weakly in $H^2(\W)$, where $u$ is the strong solution to \eqref{eqn1.1} with $b,c\equiv 0$.  Moreover, we have the following $H^1$ stability result for $\mL$:
\begin{align}\label{eqn3.12a}
\|\grad u\|_{L^2(\W)} \lss \|\mL u\|_{H^{-1}(\W)} . 
\end{align}
\end{theorem}

\begin{proof}
Since $\mLe u^\e = f$ in $\W$ and $u^\e = 0$ on $\partial\W$, we have the boundedness 
of $\| u^\e \|_{H^2(\W)}$ from the Poincar\'e inequality and Lemma \ref{lem3.5}.  By compactness we can  extract a subsequence $\{u^\e\}$ (not relabeled) and 
$u^*\in H^2(\W)\cap H_0^1(\W)$ such that $u^\e\wto u^*$ weakly in $H^2(\W)$.  
Moreover, since $\mLe u^\e = f$, we have for any $\phi\in C_0^\infty(\W)$
\begin{align} \label{eqn3.12}
\e(\Delta u^\e,\Delta\phi) - (A:D^2u^\e,\phi) = (f,\phi).
\end{align}
Since $\e(\Delta u^\e,\Delta\phi) \leq \e\|\Delta u^\e \|_{L^2(\W)}\|\Delta\phi \|_{L^2(\W)}\to 0$ as $\e\to 0$, letting  $\e\to 0$ in \eqref{eqn3.12} we obtain
\begin{align} \label{eqn3.13}
- (A:D^2u^*,\phi) = (f,\phi) \qquad \forall \phi\in C_0^\infty(\W).
\end{align}
Thus, $u^*$ is a strong solution to \eqref{eqn1.1}.  Since $\mL$ is 1-1, we have $u^*=u$, and the whole sequence $u^\e\wto u$ as $\e \to 0$.  

We now derive \eqref{eqn3.12a}.  Since $\|\mLe u^\e\|_{H^{-1}(\W)}=\|f\|_{H^{-1}(\W)}$ is constant with respect to $\e$, we take the $\liminf$ of \eqref{eqn3.11a} and use $u^\e\wto u$ to obtain 
\begin{align*}
\|\mL u\|_{H^{-1}(\W)} &=\liminf_{\e\to 0} \|\mLe u^\e\|_{H^{-1}(\W)}  \\
&\gss \liminf_{\e\to 0}\left(\sqrt{\e}\|\Delta u^\e \|_{L^2(\W)} + \sqrt{\lambda}\| \grad u^\e \|_{L^2(\W)}\right) \\
&\gss \liminf_{\e\to 0}\sqrt{\e}\|\Delta u^\e \|_{L^2(\W)} + \liminf_{\e\to 0}\sqrt{\lambda}\| \grad u^\e \|_{L^2(\W)} \\
&\gss \| \grad u\|_{L^2(\W)} . 
\end{align*}
The proof is complete.
\end{proof}

Now that we have the existence and convergence of $u^\e$ to the strong solution $u$ of \eqref{eqn1.1}, in $H^2(\W)$, 
we note that indeed the regularity of $u^\e$ can be higher than $u$ -- up to $H^4$ dependent on the smoothness 
of $\partial\W$ -- and is shown in the theorem below.

\begin{theorem}\label{thm_ue_reg}
Let $\e<\e_0$, $\partial\W\in C^3$, and $u^\e\in H$ be the weak solution to \eqref{eqn1.4}.  Then $u^\e\in H^3(\W)$ with
\begin{align} \label{eqn3.30}
\e\|u^\e\|_{H^3(\W)} \lss \|f\|_{L^2(\W)}.
\end{align}
Moreover, if $\partial\W\in C^4$, then $u^\e\in H^4(\W)$ with the estimate
\begin{align} \label{eqn3.30a}
\e\|u^\e\|_{H^4(\W)} \lss \|f\|_{L^2(\W)}.
\end{align}
\end{theorem}
\begin{proof}
Since $u^{\e}\in H^2(\W)\cap H_0^1(\W)$ and $A\in [L^{\infty}(\W)]^{d\times d}$, then $A:D^2u^\e+f\in L^2(\W)$ with $\| A:D^2u^\e+f \|_{L^2(\W)}\lss \|f\|_{L^2(\W)}$.  Moreover, since $u^\e\in H$ is weak solution to \eqref{eqn1.4}, we have
\begin{align*}
 \e(\Delta u^\e,\Delta v) = (A:D^2 u^\e+f,v) 
\end{align*}
Thus $u^{\e}$ is the weak solution to the biharmonic equation
\begin{align*}
\begin{split} 
\e\Delta^2u^\e &= A_0:D^2u^\e + f \qquad \text{ in } \W, \\ 
u^\e &= \Delta u^\e = 0 \qquad\qquad\ \text{ on } \partial\W,
\end{split}
\end{align*}
and the smoothness of $u^\e$ is entirely dependent on the smoothness of $\partial\W$,  and the desired 
estimates \eqref{eqn3.30}-\eqref{eqn3.30a} follow from the elliptic theory for the biharmonic equation. The proof is complete.
\end{proof}

We conclude this section with some error estimates of $u^\e-u$ in various norms.

\begin{theorem} \label{thm_eps_rates}
Let $\e<\e_0$, $u\in H^2(\W)\cap H_0^1(\W)$, and $u^\e\in H$ be the solutions to \eqref{eqn1.1} and \eqref{eqn1.4}, respectively.  Then we have the following error estimates: 
\begin{align} \label{eqn3.14}
\|\grad(u^\e-u) \|_{L^2(\W)} &\lss \sqrt{\e}\|f \|_{L^2(\W)} , \\
\|u^\e-u \|_{L^2(\W)} &\lss \sqrt{\e}\|f \|_{L^2(\W)} . \label{eqn3.15}
\end{align} 
\end{theorem}

\begin{proof}
Let $e^\e = u^\e - u$.  By the linearity of $\mL$ we get $\mL e^\e = \e\Delta^2u^\e\in H^{-1}(\W)$ since $u^\e\in H$.  Using \eqref{eqn3.12a} and \eqref{eqn3.11stab}, we have 
\begin{align*}
\begin{split}
\| \grad e^\e\|_{L^2(\W)} &\lss \e\| \Delta^2u^\e \|_{H^{-1}(\W)} \\
&= \sup_{v\in H_0^1(\W)}\frac{-\e(\grad\Delta u^\e,\grad v)}{\| \grad v\|_{L^2(\W)}} \\
&\leq \e \|\grad\Delta u^\e \|_{L^2(\W)} \\
&\lss \sqrt{\e} \|f \|_{L^2(\W)}, 
\end{split}
\end{align*}
which infers \eqref{eqn3.14}.  Poincar\'e's inequality yields \eqref{eqn3.15}.
The proof is complete.
\end{proof}

We note that the rates in Theorem~\ref{thm_eps_rates} may be suboptimal based on the 
numerical experiments provided below.  

\section{$C^1$ Finite Element Methods for Second Order Elliptic Linear PDEs in Non-Divergence Form}\label{section4}

Since the highest order derivative of \eqref{eqn1.4} is in divergence form, we can easily and naturally discretize it using $C^1$ conforming finite element methods. 

\begin{defn} \label{defnFEmethod}
We define our $C^1$ finite element method for \eqref{eqn1.4} as seeking $u_h^\e\in V_h$ such that
\begin{align} \label{eqnc1fe}
\mAe(u_h^\e,v_h) = (f,v_h)  \qquad\forall v_h\in V_h , 
\end{align}
where
\[
	\mAe(w_h,v_h) := \e \bigl(  \Delta w_h, \Delta v_h \bigr) - \bigl( A:D^2w_h,v_h \bigr).
\]
\end{defn} 

\subsection{Well-posedness of $\mLhe$}

In this subsection, we show the well-posedness of Problem \eqref{defnFEmethod}.  We define the operator $\mLhe:V_h\to V_h$ by 
\[
	(\mLhe v_h,w_h) = \mA^\e(v_h,w_h) \qquad\forall w_h\in V_h.  
\]
We can also naturally extend the domain of $\mLhe$ to $H^2(\W)\cap H_0^1(\W)$.  

We assume that there is a unique weak solution $u^\e\in H^2(\W)\cap H_0^1(\W)$ of \eqref{eqn1.4} and that the operator $\mLe$ satisfies the Calderon-Zygmond estimate 
\begin{align} \label{eqnnewstab}
\|v\|_{H^2(\W)} \lss \|\mLe v\|_{L^2(\W)} . 
\end{align}
We note that Theorem \ref{thm:exist} implies this result provided $\partial\W\in C^{2,1}$ and $\e$ is sufficiently small.

Our goal is to derive a similar estimate to \eqref{eqnnewstab} for $\mLhe$, that is,
\begin{align}\label{eqnfecz}
\|v_h \|_{H^2(\W)} \lss \| \mLhe v_h \|_{L_h^2(\W)} \qquad \forall v_h\in V_h.
\end{align}
To this end, we will adapt the freezing coefficient technique as in Section \ref{subsection3.1} but
at the discrete level.  Section \ref{subsection-dis-constant} shows \eqref{eqnfecz} for constant coefficient $A$ which will be used in Section \ref{subsection-dis-continuous} to give us \eqref{eqnfecz}.  

\subsubsection{Stability Analysis of $\mLhe$ for Constant Coefficient $A$}\label{subsection-dis-constant}

Suppose that $A= A_0 \equiv A(x_0)$ for an $x_0\in \W$.  Since $A$ is constant, we 
have $A:D^2 v_h = \div(A\grad v_h)$ for all $v_h\in V_h$.   Thus, we can define $\mLzhe:V_h\to V_h$ by
\begin{align*}
\bigl(\mLzhe v_h,w_h\bigr) = \mAze(v_h,w_h) \qquad\forall w_h\in V_h,
\end{align*}  
where
\begin{align*}
\mAze(v_h,w_h) := \e\bigl(\Delta v_h,\Delta w_h\bigr) + \bigl(A\grad v_h,\grad w_h\bigr).
\end{align*}  

Since $\mAze(v_h,w_h)$ is clearly continuous and coercive on $V_h$ with the norm 
$$\|w_h\|_{E}^2 := \e\|\Delta w_h\|_{L^2(\W)}^2 + \|\grad w_h\|_{L^2(\W)}^2,$$
we have the existence and uniqueness of a solution $u_h^\e$ satisfying
\begin{align} \label{eqn4.0.1}
\bigl(\mLzhe u_h^\e,w_h\bigr) = (f,w_h) \qquad\forall w_h\in V_h 
\end{align} 
and the estimate
\begin{align} \label{eqn4.0.2}
\|u_h^\e-u^\e \|_{E} \lss h^{r-2}(\sqrt{\e}+h)\|u^\e\|_{H^s(\W)} , 
\end{align}
where $r=\min\{s,k+1\}$ and $u^\e\in H^s(\W)$ is the solution to
\begin{align} \label{eqn4.0.3}
	\mAze(u^\e,v) = (f,w) \quad\forall w\in H^2(\W)\cap H_0^1(\W).  
\end{align}
With all of the above in place, we can now show \eqref{eqnfecz} for $\mLzhe$.  

\begin{theorem}\label{thm:cc-caldzyg}
Suppose $\e=\mathcal{O}(h^\beta)$ with $\beta\geq 2$. Then there holds
\begin{align}\label{eqn:cc-caldzyg}
\| w_h \|_{H^2(\W)} \lss \|\mLzhe w_h\|_{L^2(\W)} \qquad\forall w_h\in V_h.
\end{align}
\end{theorem}  

\begin{proof}
We first note that \eqref{eqn:cc-caldzyg} is equivalent to 
\begin{align} \label{eqn4.0.0.3}
\|(\mLzhe)^{-1}\phi_h \|_{H^2(\W)} \lss \|\phi_h\|_{L^2(\W)} \qquad\forall \phi_h\in V_h.
\end{align}
Fix $\phi_h\in V_h$, and let $w^\e$ be the weak solution to \eqref{eqn4.0.3} with $f=\phi_h$ and $w_h=(\mLzhe)^{-1}\phi_h\in V_h$.  Since $w^\e$ is a weak solution of \eqref{eqn4.0.3}, we have $w^\e\in H^2(\W)\cap H_0^1(\W)$ with 
\begin{align} \label{eqn4.0.4}
	\| w^\e \|_{H^2(\W)} \lss \|\phi_h\|_{L^2(\W)}
\end{align}
by \eqref{eqnnewstab}.   Since $w_h^\e$ satisfies
\[
	\mAze(w_h^\e,v_h) = (\phi_h,v_h) \qquad\forall v_h\in V_h , 
\]
we can apply \eqref{eqn4.0.2} to get  
\begin{align} \label{eqn4.0.5}
\|\grad (w^\e-w_h^\e)\|_{L^2(\W)} \leq \| w^\e-w_h^\e\|_{E} \lss (\sqrt{\e}+h)\|w^\e\|_{H^2(\W)} \lss h\|\phi_h\|_{L^2(\W)} . 
\end{align}  
Let $I_h^1$ denote the $P_1$ continuous finite element interpolation operator on $\mathcal{T}_h$. 
By the finite element interpolation theory, we have
\begin{align} \label{eqn4.0.6}
\|\grad (w^\e-I_h^1 w^\e)\|_{L^2(\W)} \lss h\|w^\e\|_{H^2(\W)} \lss h\|\phi_h\|_{L^2(\W)} , 
\end{align}  

Let $\|\cdot\|_{H^2(\mathcal{T}_h)}$ denote the piecewise $H^2$-norm on $\mathcal{T}_h$. From \eqref{eqn4.0.5}, \eqref{eqn4.0.6}, the triangle inequality, and an inverse inequality, we have
\begin{align*}
\|w^\e-w_h^\e\|_{H^2(\W)} &\lss \|w^\e-I_h^1 w^\e\|_{H^2(\mathcal{T}_h)} + \|I_h^1 w^\e-w_h^\e\|_{H^2(\mathcal{T}_h)} \\
& \lss \|w^\e\|_{H^2(\W)} + \frac{1}{h} \|\grad(I_h^1 w^\e-w_h^\e)\|_{L^2(\W)} \\
& \lss \|\phi_h\|_{L^2(\W)} + \frac{1}{h}\|\grad(I_h^1 w^\e-w^\e)\|_{L^2(\W)} 
+ \frac{1}{h}\|\grad(w^\e-w_h^\e)\|_{L^2(\W)} \\
& \lss \|\phi_h\|_{L^2(\W)}.
\end{align*}
Hence,
\begin{align*}
	\|w_h^\e\|_{H^2(\W)} \leq \|w^\e-w_h^\e\|_{H^2(\W)} + \|w^\e\|_{H^2(\W)}  \lss \|\phi_h\|_{L^2(\W)}.
\end{align*}
Since $w_h^\e = (\mLzhe)^{-1}\phi_h$, we have \eqref{eqn4.0.2}.  The proof is complete.  
\end{proof}

\subsubsection{Stability Analysis of $\mLhe$ for Continuous Coefficient $A$}
\label{subsection-dis-continuous}

We need some estimates before continuing with the freezing coefficient argument.  The first is 
to show $\mLhe$ is bounded independent of $h$ which is a result of the following lemma.  

\begin{lemma} \label{lemfe1}
Let $\e=\mathcal{O}(h^\beta)$ for $\beta\geq 2$ and $D\subseteq\W$ be a subdomain.  
Then, for any $v\in H^2(D)\cap H_0^1(D)$, there holds 
\begin{align}\label{eqn4.1}
\|\mLhe v\|_{L_h^2(D)} \lss \|v\|_{H^2(D)}.
\end{align}
\end{lemma}

\begin{proof}
Let $w_h\in V_h(D)$.  By H\"older's inequality and an inverse inequality we have
\begin{align}\label{eqn4.1.1}
\bigl(\mLhe v,w_h\bigr) &= \mAe(v,w_h) = \e(\Delta v,\Delta w_h) -(A:D^2v,w_h) \\
&\lss h^\beta\|v\|_{H^2(D)}\|w_h\|_{H^2(D)} + \|A\|_{L^{\infty}(D)}\|v\|_{H^2(D)}\|w_h\|_{L^2(D)} \nonumber \\
&\lss h^{\beta-2}\|v\|_{H^2(D)}\|w_h\|_{L^2(D)} + \|v\|_{H^2(D)}\|w_h\|_{L^2(D)} \nonumber \\
&\lss \|v\|_{H^2(D)}\|w_h\|_{L^2(D)} \nonumber.
\end{align}
Diving both sides of \eqref{eqn4.1.1} by $\|w_h\|_{L^2(D)}$ yields \eqref{eqn4.1}.  The proof is complete.  
\end{proof}

Next we need to show that $\mLzhe$ and $\mLhe$ are locally close to each other.  

\begin{lemma} \label{lemfe2}
Let $x_0\in \W$ and let $\d>0$.  Then there exists $R_{\d}>0$ and $h_{\d} >0$  such that  
\begin{align}\label{eqn4.2}
\|(\mLzhe-\mLhe) v_h\|_{L_h^2(B_{R_{\d}})} \lss \d\|v_h\|_{H^2(B_{R_{\d}})} \quad\forall v_h\in V_h(B_{R_{\d}}),
\end{align}
where $B_{R_{\d}}$ is the ball of radius ${R_{\d}}$ centered at $x_0$ and $h<h_{\d}$.  
\end{lemma}

\begin{proof}
Since $A$ is continuous and $\overline\W$ is compact, it is uniformly continuous on $\overline\W$.  
Thus, given $\d>0$, there exists $R_{\d}>0$ such that 
\begin{align} \label{eqn4.2.1}
\| A_0- A \|_{L^{\infty}(B_{R_{\d}})} \leq \delta.
\end{align}
Choose $h_{\d}$ such that $V_h(B_{R_{\d}})$ is nontrivial.  Let $w_h\in V_h(B_{R_{\d}})$ with $w_h\neq 0$.  Then \begin{align*}
\left((\mLzhe-\mLhe) v_h,w_h\right) &= \left((A_0-A):D^2 v_h,w_h\right) \\
&\lss \|A_0-A\|_{L^{\infty}(B_{R_{\d}})}\|v_h\|_{H^2(B_{R_{\d}})}\|w_h\|_{L^2(B_{R_{\d}})}.
\end{align*}
Dividing by $\|w_h\|_{L^2(B)}$ gives us \eqref{eqn4.2}.  The proof is complete. 
\end{proof}

Lastly, we state some super approximation results similar to those given in \cite{Feng2017SS}.  Since the proof is almost the same by verifying the estimates on each element, we omit it.

\begin{lemma}\label{lemfesuperapprox}
Let $I_h$ be a standard $C^1$ (e.g. Argyris) nodal finite element interpolation operator (cf. \cite{Sp:BS}) and 
 $\eta\in C^{\infty}(\W)$ with $\|\eta\|_{W^{j,\infty}(\W)}\lss d^{-j}$ for $0\leq j\leq k$ and some $d>0$.  Then, for any subdomain $D\subseteq\W$, we have
\begin{align*}
\|\eta v_h - I_h(\eta v_h) \|_{L^2(D)} &\lss \frac{h}{d}\|v_h\|_{L^2(D)} \qquad\forall v_h\in V_h,\\
\|\eta v_h - I_h(\eta v_h) \|_{H^2(D)} &\lss \frac{h}{d^3}\|v_h\|_{H^2(D)} \qquad\forall v_h\in V_h.
\end{align*}
\end{lemma}

We also cite the following inverse inequality from \cite{Feng2017SS}.

\begin{lemma}\label{leminverseineq}
Let $v_h\in V_h(D)$ where $D\subseteq\W$ is a subdomain.  Then we have
\begin{align*}
\|v_h\|_{L^2(D)} \lss h^{-1}\|v_h\|_{H^{-1}(D)}.
\end{align*}
\end{lemma}

To follow the freezing coefficient technique from here, we could employ arguments similar
 to the likes of Section  \ref{section3.2.2} to achieve a G\r{a}rding-type estimate similar to \eqref{eqn3.11}; namely, 
\begin{align}
 \| v_h \|_{H^2(\W)} \lss \| \mLhe v_h \|_{L_h^2(\W)} + \|v_h\|_{L^2(\W)}, 
\end{align}
which is almost \eqref{eqnfecz} aside from the $L^2$ norm of $v_h$ on the right hand side.  To strip this term off in the PDE theory, we used the fact that $\mLe$ was injective.  However, we do not have the tools available at the discrete level to say the same for $\mLhe$.  To overcome this difficulty, our idea is to utilize  a duality argument in order to achieve \eqref{eqnfecz}.  Unfortunately, a standard duality argument would require an a-priori stability estimate of $(\mLe)^*$, the adjoint of $\mLe$, which is unknown at this time.  Instead, we apply the freezing coefficient technique to $ (\mLhe)^*$, the adjoint of $\mLhe$, for two reasons.  First, since $\mLhe$ lives on a finite dimensional vector space, the invertibility of $\mLhe$ is equivalent to the invertibility of $(\mLhe)^*$.  Second, a duality argument applied to $ (\mLhe)^*$ uses the stability of $\mLe$, which we have already proved in Theorem \ref{thm:exist}. 

Define $\mLhes:V_h\to V_h$ by 
\[
\bigl(\mLhes w_h,v_h\bigl) = (\mLhe v_h,w_h) \quad\forall v_h\in V_h.   
\]
The entirety of the non-standard duality argument is given in the lemma below.

\begin{lemma} \label{lemfe3}
There exists $h_0>0$ such that for all $h<h_0$ and $\e=\mathcal{O}(h^{\beta})$ for some $\beta\geq 2$, 
\begin{align}\label{eqn4.3}
\|v_h\|_{L^2(\W)} \lss \|\mLhes v_h\|_{H_h^{-2}(\W)}  \qquad \forall v_h\in V_h.
\end{align}
\end{lemma}

\begin{proof}  
Since the proof is long, we divide it into three steps. 

{\em Step 1: Local Estimates for $\mLhes$.}
Let $x_0\in \W$.  Given $\d>0$, define $h_{\d}$ and $R_{\d}$ as in Lemma \ref{lemfe2}.  Let $h<h_{\d}$, $R_1 = (1/3)R_{\d}$, and $B_1$ be the ball centered at $x_0$ with radius $R_1$.  Let $v_h\in V_h(B_{R_1})$. From the assumption on $\mLe$, there exists $\phi^\e\in H^2(\W)\cap H_0^1(\W)$ with 
\[
\mAze(\phi^\e,w) = (v_h^\e,w)\quad\forall w\in H^2(\W)\cap H_0^1(\W)
\] and 
\begin{align} \label{eqn4.3.1}
\|\phi^\e\|_{H^2(\W)} \lss \|v_h\|_{L^2(\W)} = \|v_h\|_{L^2(B_1)}.
\end{align}
We also note that 
\begin{align} \label{eqn4.3.2}
\bigl(\mLhe \psi,w_h\bigr) = \mAe(\psi,w_h) = \bigl(\mLe \psi,w_h\bigr) \qquad\forall 
\psi\in H^2(\W)\cap H_0^1(\W),\, w_h\in V_h.
\end{align}
Let $\phi_h^\e = (\mLzhe)^{-1}v_h \in V_h$.  By \eqref{eqn4.3.2}, we have
\begin{align}\label{eqn4.3.3}
\|v_h\|_{L^2(B_1)}^2 &= \|v_h\|_{L^2(\W)} = (v_h,v_h) = (\mLe \phi^\e,v_h) = (\mLhe \phi^\e,v_h)\\
&=\bigl(\mLhe\phi_h^\e,v_h\bigr) + \bigl(\mLhe(\phi^\e-\phi_h^\e),v_h\bigr) \nonumber\\
&= \bigl(\mLhes v_h,\phi_h^\e\bigr) + \bigl(\mLzhe(\phi^\e-\phi_h^\e),v_h\bigr) \nonumber\\
&\qquad + \bigl((\mLhe-\mLzhe)(\phi^\e-\phi_h^\e),v_h\bigr).  \nonumber
\end{align}
Since $\phi_h^\e = (\mLzhe)^{-1}v_h$, we obtain the standard Galerkin orthogonality, namely
\begin{align} \label{eqn4.3.4}
(\mLzhe(\phi^\e-\phi_h^\e),v_h) &= \mAze(\phi^\e-\phi_h^\e,v_h) = \mAze(\phi^\e,v_h) - \mAze(\phi_h^\e,v_h) \\
&= (v_h,v_h) - (v_h,v_h) = 0.\nonumber
\end{align}
Moreover, by Theorem \ref{thm:cc-caldzyg}, \eqref{eqn4.3.4}, and \eqref{eqn4.3.1}, we have
\begin{align}\label{eqn4.3.5}
\|\phi_h^\e\|_{H^2(\W)} &\lss \|\mLzhe \phi_h^\e\|_{L_h^2(\W)} \leq \|\mLzhe \phi^\e\|_{L_h^2(\W)} = \|\mLzhe \phi^\e\|_{L^2(\W)} \\ 
& \lss  \|\phi^\e\|_{H^2(\W)} \lss \|v_h\|_{L^2(B_1)}.   \nonumber
\end{align}
Employing Lemma \ref{lemfe2}, \eqref{eqn4.3.1}, \eqref{eqn4.3.3}, and \eqref{eqn4.3.5}, we obtain
\begin{align*}
\|v_h\|_{L^2(B_1)}^2 &= (\mLhes v_h,\phi_h^\e) + ((\mLhe-\mLzhe)(\phi^\e-\phi_h^\e),v_h) \\
&\leq \|\mLhes v_h\|_{H_h^{-2}(\W)}\|\phi_h^\e\|_{H^2(\W)} + \|(\mLhe-\mLzhe)(\phi^\e-\phi_h^\e)\|_{L_h^2(\W)}\|v_h\|_{L^2(\W)} \\
&\lss \|\mLhes v_h\|_{H_h^{-2}(B_1)}\|v_h\|_{L^2(B_1)} + \d\|(\phi^\e-\phi_h^\e)\|_{H^2(\W)}\|v_h\|_{L^2(B_1)} \\
&\lss \|\mLhes v_h\|_{H_h^{-2}(B_1)}\|v_h\|_{L^2(B_1)} + \d(\|\phi^\e\|_{H^2(\W)}+\|\phi_h^\e\|_{H^2(\W)})\|v_h\|_{L^2(B_1)} \\
&\lss \|\mLhes v_h\|_{H_h^{-2}(B_1)}\|v_h\|_{L^2(B_1)} + \d\|v_h\|_{L^2(B_1)}^2 . 
\end{align*}
Taking $\d$ sufficiently small and then dividing both sides by $\|v_h\|_{L^2(B_1)}$ 
gives us a local version of \eqref{eqn4.3}, that is,
\begin{align}\label{eqn4.3.6}
\|v_h\|_{L^2(B_1)} \lss \|\mLhes v_h\|_{H_h^{-2}(B_1)}.
\end{align}

{\em Step 2: A G\r{a}rding-type inequality.}
Let $R_3 = 3R_1$, and $\eta\in C^\infty(\W)$ be a cutoff function with the following properties:
\begin{align} \label{eqn4.3.6.1}
0\leq \eta\leq 1, \quad \eta\big|_{B_1} \equiv 1,\quad \eta\big|_{\W\setminus B_2} \equiv 0,\quad |\eta|_{W^{m,\infty}(\W)} = \mathcal{O}(R_1^{-m}).  
\end{align}
Let $v_h\in V_h$.  By \eqref{eqn4.3.6} we have
\begin{align}\label{eqn4.3.7}
\|v_h\|_{L^2(B_1)} &= \|\eta v_h\|_{L^2(B_1)} \leq \|\eta v_h- I_h(\eta v_h)\|_{L^2(B_1)} + \|I_h(\eta v_h)\|_{L^2(B_1)} \\
&\lss \|\eta v_h- I_h(\eta v_h)\|_{L^2(B_1)} + \|\mLhes I_h(\eta v_h)\|_{H_h^{-2}(B_1)} \nonumber\\
&\lss \|\eta v_h- I_h(\eta v_h)\|_{L^2(B_1)} + \|\mLhes (I_h(\eta v_h)-\eta v_h)\|_{H_h^{-2}(B_1)} 
\nonumber \\
&\qquad + \|\mLhes\eta v_h\|_{H_h^{-2}(B_1)} \nonumber.
\end{align}
Using Lemma \ref{lemfe1}, we have
\begin{align*}
\|\mLhes (I_h(\eta v_h)-\eta v_h)\|_{H_h^{-2}(B_1)} &= \sup_{w_h\in V_h(B_1)} \frac{(\mLhes (I_h(\eta v_h)-\eta v_h),w_h)}{\|w_h\|_{H^2(B_1)}} \\ 
&\lss \sup_{w_h\in V_h(B_1)} \frac{(\mLhe w_h,(I_h(\eta v_h)-\eta v_h))}{\|w_h\|_{H^2(B_1)}}
\nonumber \\
&\lss \sup_{w_h\in V_h(B_1)} \frac{\|\mLhe w_h\|_{L^2(B_1)}\|I_h(\eta v_h)-\eta v_h\|_{L^2(B_1)}}{\|w_h\|_{H^2(B_1)}}  \nonumber \\
&\lss \sup_{w_h\in V_h(B_1)} \frac{\|w_h\|_{H^2(B_1)}\|I_h(\eta v_h)-\eta v_h\|_{L^2(B_1)}}{\|w_h\|_{H^2(B_1)}} \nonumber\\
& = \|I_h(\eta v_h)-\eta v_h\|_{L^2(B_1)}. \nonumber 
\end{align*}
Thus, \eqref{eqn4.3.7} becomes
\begin{align}\label{eqn4.3.8}
\|v_h\|_{L^2(B_1)} \lss \|\eta v_h- I_h(\eta v_h)\|_{L^2(B_1)} + \|\mLhes\eta v_h\|_{H_h^{-2}(B_1)}.
\end{align}
Using Lemmas \ref{lemfesuperapprox} and \ref{leminverseineq}, we have 
\begin{align}\label{eqn4.3.9}
\|v_h\|_{L^2(B_1)} &\lss \frac{h}{R_1}\|v_h\|_{L^2(B_3)} +  \|\mLhes\eta v_h\|_{H_h^{-2}(B_3)} \\
&\lss\frac{1}{R_1}\|v_h\|_{H^{-1}(B_3)} +  \|\mLhes\eta v_h\|_{H_h^{-2}(B_3)}. \nonumber
\end{align}
We now wish to remove $\eta$ from $\|\mLhes\eta v_h\|_{H_h^{-2}(B_3)}$ in \eqref{eqn4.3.9}.  Let $w_h\in V_h(B_3)$ and decompose $\mLhes\eta v_h$ as follows:

\begin{align}\label{eqn4.3.10}
\bigl(\mLhes( \eta v_h),w_h\bigr) &= \bigl(\mLhe w_h,\eta v_h\bigr) \\
&= \bigl(\mLhe w_h\eta,v_h\bigr)
+\big[ \bigl(\mLhe w_h,\eta v_h\bigr)-\bigl(\mLhe w_h\eta, v\bigr)\big] \nonumber \\
&= \bigl(\mLhe (I_h(w_h\eta)),v_h\bigr)+\bigl(\mLhe(w_h\eta-I_h(w_h\eta)),v_h\bigr) \nonumber \\
&\quad+\big[(\mLhe w_h,\eta v_h)-(\mLhe w_h\eta, v_h)\big]\nonumber \\
&=: I_1 + I_2 + I_3 \nonumber.
\end{align}
We now bound each $I_i$ in order.  For $I_1$, the stability of $I_h$ and \eqref{eqn4.3.6.1} give us
\begin{align}\label{eqn4.3.11}
I_1 &= \bigl(\mLhes v_h,I_h(\eta w_h)\bigr)\lss \|\mLhes v_h\|_{H_h^{-2}(B_3)}\|I_h(\eta w_h)\|_{H^2(B_3)} \\ 
&\lss\|\mLhes v_h\|_{H_h^{-2}(B_3)}\|\eta w_h\|_{H^2(B_3)} \nonumber \\
&\lss \frac{1}{R_1^2}\|\mLhes v_h\|_{H_h^{-2}(B_3)}\|w_h\|_{H^2(B_3)}.  \nonumber
\end{align}
To bound $I_2$, we use Lemmas \ref{lemfesuperapprox}, \ref{leminverseineq}, 
and \ref{lemfe1} to obtain
\begin{align}\label{eqn4.3.12}
I_2 &\lss \|\eta w_h - I_h(\eta w_h) \|_{H^2(B_3)}\| v_h \|_{L^2(B_3)}\lss \frac{h}{R_1^3}\|w_h\|_{H^2(B_3)}\|v_h\|_{L^2(B_3)} \\
&\lss  \frac{1}{R_1^3}\|w_h\|_{H^2(B_3)}\|v_h\|_{H^{-1}(B_3)}. \nonumber
\end{align}
In order to bound $I_3$, we reintroduce the operator $\mLzhe$ and define $\tilde{A}=A-A_0$.  We then write
\begin{align}\label{eqn4.3.13}
I_3 &= \bigl(\mLhe w_h,\eta v_h\bigr) - \bigl(\mLhe w_h \eta,v_h\bigr) \\
&=\bigl(\mLzhe w_h,\eta v_h\bigr) - \bigl(\mLzhe \eta w_h,v_h\bigr) \nonumber \\
&\quad+\Bigl[ \bigl(\mLhe w_h,\eta v_h\bigr) - \bigl(\mLhe \eta w_h,v_h\bigr) 
- \bigl(\mLzhe w_h,\eta v_h\bigr) + \bigl(\mLzhe \eta w_h,v_h\bigr) \Bigr]\nonumber \\
&=\e\int_{B_3} v_h\Delta w_h\Delta \eta  + 2\Delta w_h(\grad \eta\cdot\grad v_h)-w_h\Delta \eta\Delta v_h - 2\Delta w_h(\grad \eta\cdot\grad w_h)\dx[x] \nonumber \\
&\quad -\int_{B_3} \left(w_hA_0:D^2\eta + (A_0+A_0^T)\grad\eta\cdot\grad w_h \right) v_h\dx[x]\nonumber \\
&\quad -\int_{B_3} \left(w_h\tilde{A}:D^2\eta + (\tilde{A}+\tilde{A}^T)\grad\eta\cdot\grad w_h \right) v_h\dx[x]\nonumber=:K_1 + K_2 + K_3.
\end{align}
We now must bound each $K_i$.  Since $\e<Ch^2$ for some pure constant $C>0$, we bound $K_1$ using Lemma \ref{leminverseineq}, standard inverse inequalities, and \eqref{eqn4.3.6.1}
as follows:
\begin{align}\label{eqn4.3.14}
K_1 &\lss \frac{h^2}{R_1^2} \|v_h\|_{L^2(B_3)}\|w_h\|_{H^2(B_3)} + \frac{h_2}{R_1}\|v_h\|_{H^1(B_3)}\|w_h\|_{H^2(B_3)} \\ 
&\quad + h^2\|\Delta v_h\|_{H^{-1}(B_3)}\|w_h\Delta\eta\|_{H^{1}(B_3)} +  h^2\|\Delta v_h\|_{H^{-1}(B_3)}\|\grad w_h\cdot\grad\eta\|_{H^{1}(B_3)}\nonumber \\
&\lss \frac{h}{R_1^2} \|v_h\|_{H^{-1}(B_3)}\|w_h\|_{H^2(B_3)} + \frac{1}{R_1}\|v_h\|_{H^{-1}(B_3)}\|w_h\|_{H^2(B_3)} \nonumber\\
&\quad + \frac{h^2}{R_1^3}\|\grad v_h\|_{L^2(B_3)}\|w_h\|_{H^2(B_3)} +  \frac{h^2}{R_1^2}\|\grad v_h\|_{L^2(B_3)}\|w_h\|_{H^2(B_3)}\nonumber \\
&\lss \frac{h}{R_1^2} \|v_h\|_{H^{-1}(B_3)}\|w_h\|_{H^2(B_3)} + \frac{1}{R_1}\|v_h\|_{H^{-1}(B_3)}\|w_h\|_{H^2(B_3)} \nonumber\\
&\quad + \frac{1}{R_1^3}\|v_h\|_{H^{-1}(B_3)}\|w_h\|_{H^2(B_3)} +  \frac{1}{R_1^2}\|v_h\|_{H^{-1}(B_3)}\|w_h\|_{H^2(B_3)}\nonumber \\
&\lss \frac{1}{R_1^3} \|w_h\|_{H^2(B_3)}\|v_h\|_{H^{-1}(B_3)}. \nonumber 
\end{align} 
To bound $K_2$, we have 
\begin{align}\label{eqn4.3.15}
K_2 &\lss \Bigl(\|w_hA_0:D^2\eta\|_{H^1(B_3)}
+ \|(A_0+A_0^T)\grad\eta\cdot\grad w_h\_{H^1(B_3)}\Bigr)\|v_h\|_{H^{-1}(B_3)} \\
&\lss \frac{1}{R_1^3}\|w_h\|_{H^2(B_3)}\|v_h\|_{H^{-1}(B_3)}. \nonumber
\end{align}
Bounding $K_3$ requires H\"older's inequality and $\|\tilde{A}\|_{L^{\infty}(B_3)}\leq \d$ which gives us
\begin{align}\label{eqn4.3.16}
K_3 &\lss \Bigl(\|w_h\tilde{A}:D^2\eta\|_{L^2(B_3)}
+ \|(\tilde{A}+\tilde{A}^T)\grad\eta\cdot\grad w_h\|_{L^2(B_3)}\Bigr)\|v_h\|_{L^2(B_3)} \\
&\lss \d\left(\frac{1}{R_1^2}\|w_h\|_{L^2(B_3)}+\frac{1}{R_1}\|w_h\|_{H^1(B_3)}\right)\|v_h\|_{L^2(B_3)}. \nonumber
\end{align}
From Lemma C.1 of \cite{AX:FN} we have
\begin{align}\label{eqn4.3.17}
\|w_h\|_{H^l(B_3)} \lss R_1^{2-l}\|w_h\|_{H^2(B_3)},\qquad  l=0,1.
\end{align} 
Combining \eqref{eqn4.3.16} and \eqref{eqn4.3.17} yields
\begin{align}\label{eqn4.3.18}
K_3 &\lss \d \|w_h\|_{H^2(B_3)}\|v_h\|_{L^2(B_3)}. 
\end{align}
From \eqref{eqn4.3.14}, \eqref{eqn4.3.15}, and \eqref{eqn4.3.18} we get
\begin{align}\label{eqn4.3.19}
I_3 \lss \frac{1}{R_1^3}\|w_h\|_{H^2(B_3)}\|v_h\|_{H^{-1}(B_3)} + \delta \|w_h\|_{H^2(B_3)}\|v_h\|_{L^2(B_3)},
\end{align}
and from \eqref{eqn4.3.11}, \eqref{eqn4.3.12}, and \eqref{eqn4.3.19}, we have 
\begin{align*}
\bigl(\mLhes(\eta v_h),w_h \bigr) &\lss \frac{1}{R_1^3}
\Bigl( \|\mLhes v_h\|_{H^{-2}_h(B_3)} + \|v_h\|_{H^{-1}(B_3)} \Bigr) \|w_h\|_{H^2(B_3)} \\
&\quad+ \delta \|w_h\|_{H^2(B_3)}\|v_h\|_{L^2(B_3)}.
\end{align*}
Thus, 
\begin{align}\label{eqn4.3.20}
\|\mLhes(\eta v_h)\|_{H_h^{-2}(B_3)} &\lss \frac{1}{R_1^3}\Bigl( \|\mLhes v_h\|_{H^{-2}_h(B_3)}  
+ \|v_h\|_{H^{-1}(B_3)} \Bigr) \\
&\qquad + \delta\|v_h\|_{H^{-1}(B_3)}. \nonumber
\end{align}
From \eqref{eqn4.3.9} and \eqref{eqn4.3.20} we obtain
\begin{align*}
\|v_h\|_{L^2(B_1)}\lss \frac{1}{R_1^3}\left( \|\mLhes v_h\|_{H^{-2}_h(B_3)}  + \|v_h\|_{H^{-1}(B_3)} \right)+ \delta\|v_h\|_{H^{-1}(B_3)}.
\end{align*}
Since $\overline\W$ is compact and $\d$ is independent of $R_1$,
by employing a covering argument we can show that
\begin{align}\label{eqn4.3.21}
\|v_h\|_{L^2(\W)}\lss \frac{1}{R_1^3}\Bigl( \|\mLhes v_h\|_{H^{-2}_h(\W)}  
+ \|v_h\|_{H^{-1}(\W)} \Bigr)+ \delta\|v_h\|_{H^{-1}(\W)}.
\end{align}
Choosing $\d$ sufficiently small in \eqref{eqn4.3.21} produces a G\r{a}rding type estimate, namely
\begin{align}\label{eqn4.3.22}
\|v_h\|_{L^2(\W)}\lss \|\mLhes v_h\|_{H^{-2}_h(\W)}  + \|v_h\|_{H^{-1}(\W)}.
\end{align}

{\em Step 3: A nonstandard duality argument.}
We now wish to remove $\|v_h\|_{H^{-1}(\W)}$ in \eqref{eqn4.3.22}.  To control this term we 
use a nonstandard duality argument on the operator $\mLhes$.  Let
\begin{align*}
X =\bigl\{g\in H_0^1(\W): \|\grad g\|_{L^2(\W)} = 1 \bigr\}.
\end{align*}
By the Poincar\'e inequality, there exists a constant $C = C(\W)$ such that
\begin{align*}
\|g\|_{L^2(\W)} \leq C \|\grad g\|_{L^2(\W)} < \infty.  
\end{align*}
Thus, $X$ is precompact in $H^1(\W)$ by Sobolev embedding.  Now we define
\begin{align*}
W = \bigl\{ (\mLe)^{-1}g:g\in X \bigr\}.
\end{align*}
Let $\phi\in W$ with $\mLe\phi = g$.  Then by \eqref{eqnnewstab} we have
\begin{align}\label{eqn4.3.23}
\|\phi\|_{H^2(\W)} \lss \|g\|_{L^2(\W)} <\infty. 
\end{align}
Thus, $W$ is a bounded set in $H^2(\W)\cap H_0^1(\W)$.  Let $\tau>0$.  Then from \cite[Lemma 5]{MC:SW} there exists $h_*>0$ only dependent upon $\tau$ and the closure of $W$ such that for every $\phi\in W$ and $0<h\leq h_*$, there exists a $\phi_h\in V_h$ such that
\begin{align}\label{eqn4.3.24}
\|\phi-\phi_h\|_{H^2(\W)} \leq \tau.
\end{align}
Note by the triangle inequality, \eqref{eqn4.3.23}, and \eqref{eqn4.3.24}, we have
\begin{align*}
\|\phi_h\|_{H^2(\W)} \leq \|\phi\|_{H^2(\W)} + \|\phi_h-\phi\|_{H^2(\W)} \lss C + \tau. 
\end{align*}
Thus, the set
\begin{align*}
\bigl\{\phi_h\in V_h : \|\phi_h-\phi\|_{H^2(\W)}\leq \tau \bigr\}
\end{align*}
is uniformly bounded in $\phi$ and $h$.  Let $g\in X$ and choose $\phi_g\in W$ 
with $\mLe\phi_g = g$. For $v_h\in V_h$ and $\phi_h\in V_h$ satisfying \eqref{eqn4.3.24},
we have
\begin{align}\label{eqn4.3.25}
(v_h,g) &=  (g,v_h) \\ 
&= (\mLe\phi_g,v_h) = (\mLhe\phi_g,v_h) = (\mLhe\phi_h,v_h) + (\mLhe(\phi_g-\phi_h),v_h) \nonumber \\
&= (\mLhes v_h,\phi_h) +  (\mLhe(\phi_g-\phi_h),v_h) \nonumber \\
&\lss \|\mLhes v_h\|_{H^{-2}_h(\W)}\|\phi_h\|_{H^2(\W)} + \|\phi_g-\phi_h\|_{H^2(\W)}\|v_h\|_{L^2(\W)} \nonumber\\
&\lss  \|\mLhes v_h\|_{H^{-2}_h(\W)} + \tau\|v_h\|_{L^2(\W)}. \nonumber
\end{align}  
Taking the supremum of both sides of \eqref{eqn4.3.25} over $g\in X$ gives us
\begin{align}\label{eqn4.3.26}
\|v_h\|_{H^{-1}(\W)} \lss \|\mLhes v_h\|_{H^{-2}_h(\W)} + \tau\|v_h\|_{L^2(\W)}.
\end{align}
Combining \eqref{eqn4.3.22} and \eqref{eqn4.3.26} and choosing $\tau$ sufficiently small gives us \eqref{eqn4.3} upon taking $h_0 = \min\{h_{\d},h_*\}$.  
\end{proof}

Lemma \ref{lemfe3} allows us to prove the discrete Calderon-Zygmund 
estimate \eqref{eqnfecz}.  This is shown in the Lemma below.

\begin{lemma} \label{lemfe4}
Let $h_0$ be from Lemma \ref{lemfe3} and $\e=\mathcal{O}(h^\beta)$ for some $\beta\geq 2$.  Then 
\begin{align} \label{eqn4.4}
\| v_h \|_{H^2(\W)} \lss \| \mLhe v_h \|_{L_h^2(\W)} \qquad\forall   v_h\in V_h.
\end{align}
\end{lemma}

\begin{proof}
We first show $\mLhes$ is 1-1.  Suppose there is a $w_h\in V_h$ with $\mLhes w_h = 0$.  By Lemma \ref{lemfe3}, we have
\begin{align}\label{eqn4.4.0}
\| w_h \|_{L^2(\W)} \lss \|\mLhes w_h\|_{H_h^{-2}(\W)}  = 0.
\end{align}
Thus $w_h = 0$, and it follows that $\mLhes$ is 1-1.  Since $\mLhes:V_h\to V_h$, it is invertible on $V_h$.
Let $v_h\in V_h$. Consider the problem of finding $q_h\in V_h$ such that
\begin{align}
\bigl(w_h,\mLhes q_h\bigr) = \bigl(D^2v_h,D^2w_h\bigr) + (\grad v_h,\grad w_h) + (v_h,w_h). 
\end{align}
Since $\mLhes$ is invertible, such a $q_h$ exits.  By \eqref{eqn4.3} we have
\begin{align}\label{eqn4.4.1}
\|q_h\|_{L^2(\W)} &\lss \|\mLhes q_h\|_{H_h^{-2}(\W)} = \sup_{w_h\in V_h}\frac{(w_h,\mLhes q_h)}{\|w_h\|_{H^2(\W)}} \\ 
&\lss  \sup_{w_h\in V_h}\left( \frac{\|D^2v_h\|_{L^2(\W)}\|D^2w_h\|_{L^2(\W)} + \|\nabla v_h\|_{L^2(\W)}\|\nabla w_h\|_{L^2(\W)}}{\|w_h\|_{H^2(\W)}} \right. \nonumber \\ 
& \qquad \qquad \qquad + \left. \frac{\|v_h\|_{L^2(\W)}\|w_h\|_{L^2(\W)}}{\|w_h\|_{H^2(\W)}} \right) \nonumber \\
&\lss \|v_h\|_{H^2(\W)}. \nonumber 
\end{align}
From \eqref{eqn4.4.1} we obtain
\begin{align}\label{eqn4.4.2}
\| v_h \|_{H^2(\W)}^2 &= \bigl(D^2v_h,D^2v_h\bigr) + (\grad v_h,\grad v_h) + (v_h,v_h) = \bigl(v_h,\mLhes q_h\bigr) 
= \bigl(\mLhe v_h,q_h\bigr) \\
&\leq \|\mLhe v_h\|_{L_h^2(\W)}\|q_h\|_{L^2(\W)} \lss \|\mLhe v_h\|_{L_h^2(\W)}\|v_h\|_{H^2(\W)}. \nonumber
\end{align}
Dividing both sides of \eqref{eqn4.4.2} by $\|v_h\|_{H^2(\W)}$ gives us \eqref{eqn4.4}.  The proof is complete.
\end{proof}

We now present the well-posedness of our $C^1$ finite element scheme.

\begin{theorem}\label{thmfeexist}
Let $h_0$ be from Lemma \ref{lemfe3} and $\e=\mathcal{O}(h^\beta)$ for some $\beta\geq 2$. Then there is a unique $u_h^\e$ satisfying \eqref{eqnc1fe} and
\begin{align}\label{eqn4.5}
\|u_h^\e\|_{H^2(\W)} \lss \|f \|_{L^2(\W)}.
\end{align}
\end{theorem}

\begin{proof}
From a similar argument to \eqref{eqn4.4.0}, we assert that $\mLhe$ is 1-1.  Since $\mLhe:V_h\to V_h$, it is invertible.  Thus, given $f\in L^2(\W)$, there is a unique $u_h^\e\in V_h$ satisfying \eqref{eqnc1fe}.  By \eqref{eqn4.4} we have
\begin{align*}
\|u_h^\e\|_{H^2(\W)} \lss \|\mLhe u_h\|_{L_h^2(\W)} = \sup_{v_h\in V_h}\frac{(\mLhe u_h,v_h)}{\|v_h\|_{L^2(\W)}} = \sup_{v_h\in V_h}\frac{(f,v_h)}{\|v_h\|_{L^2(\W)}} \leq \|f\|_{L^2(\W)},
\end{align*}
yielding \eqref{eqn4.5}.  The proof is complete.
\end{proof}

Since $\mLhe$ is consistent, we have a C\'ea type estimate which asserts an optimal convergence in the $H^2(\W)$-norm:

\begin{theorem}\label{thmfeCea}
Suppose $\e = \mathcal{O}(h^\beta)$ for some $\beta\geq 2$.  Let $u^{\e}\in H^2(\W)\cap H_0^1(\W)$ and $u_h^\e\in V_h$ satisfy \eqref{eqnc1fe} and \eqref{eqn:weaksoln}, respectively.  Then 
\begin{align}\label{eqn4.6}
\| u^\e-u_h^\e \|_{H^2(\W)} \lss \inf_{v_h\in V_h} \| u^\e-v_h \|_{H^2(\W)}.
\end{align}
Moreover, if $u^\e\in H^s(\W)$ for $s\geq 2$, we have
\begin{align}\label{eqn4.7}
\| u^\e-u_h^\e \|_{H^2(\W)} \lss h^{r-2}\| u^\e \|_{H^r(\W)},
\end{align}
where $r=\min\{k+1,s\}$.  
\end{theorem}

\begin{proof}
By the consistency of $\mLhe$, we have the usual Galerkin orthogonality:
\begin{align}\label{eqn4.6.1}
\bigl(\mLhe(u^\e-u_h^\e),v_h\bigr) &= \mAe(u^\e-u_h^\e,v_h) \\
& = (f,v_h) - (f,v_h) = 0 \qquad\forall v_h\in V_h.  \nonumber 
\end{align}
Let $v_h\in V_h$.  By Theorem \ref{thmfeexist}, Lemma \ref{lemfe1}, and \eqref{eqn4.6.1}, we have
\begin{align}\label{eqn4.6.2}
\| u_h^\e - v_h \|_{H^2(\W)} &\lss \|\mLhe (u_h^\e - v_h) \|_{L_h^2(\W)} = \sup_{w_h\in V_h} \frac{ \bigl(\mLhe (u_h^\e - v_h),w_h\bigr)}{\|w_h\|_{L^2(\W}} \\
&\lss \sup_{w_h\in V_h} \frac{\bigl(\mLhe (u^\e - v_h),w_h\bigr)}{\|w_h\|_{L^2(\W)}}  \lss  \|\mLhe (u^\e - v_h)\|_{L_h^2(\W)}      \nonumber \\
&\lss \|u^\e-v_h\|_{H^2(\W)}.  \nonumber
\end{align}
Hence, by the triangle inequality and \eqref{eqn4.6.2}, we obtain
\begin{align}\label{eqn4.6.3}
\| u^\e - u_h^\e \|_{H^2(\W)} \lss \| u^\e - v_h \|_{H^2(\W)} + \| u_h^\e - v_h \|_{H^2(\W)} \lss \| u^\e - v_h \|_{H^2(\W)}.
\end{align}
Taking the infimum of \eqref{eqn4.6.3} over all $v_h\in V_h$ yields \eqref{eqn4.6}.  Estimate \eqref{eqn4.7} follows from taking $v_h = I_h u$ and using the standard interpolation estimates.  The proof is complete.
\end{proof}

\subsection{Numerical Experiments} \label{sect:numericaltests}

In this section, we present a series of 2-D numerical tests  
using the $C^1$-conforming (fifth order) Argyris finite element space.  
All of the tests are performed using the COMSOL software package.   
The first three tests correspond to choosing the continuous, positive definite matrix 
\begin{equation} \label{numerical_testA}
	A(x,y) = \left[
	\begin{array}{ccc}
		\left( 2x - y \right)^{1/3} + 4 e^{2-x} & \, & 
			\frac12 \sin \left( 10xy \right) - \frac12 \left( x + 2 \right)^{1/2} \\
		\frac12 \sin \left( 10xy \right) - \frac12 \left( x + 2 \right)^{1/2} & \, & 
			\left| y - 2x \right|^{1/4} + 3
	\end{array}
		\right]
\end{equation}
and $b,c = 0$ in \eqref{eqn1.1}.  
The fourth test will correspond to choosing a degenerate elliptic matrix $A$ in \eqref{eqnc1fe} 
to gauge the impact of the method for approximating ``harder" problems with 
viscosity solutions instead of strong solutions.  
We use a very fine mesh in all of the tests so that the error is dominated by the 
vanishing moment approximation to the PDE problem.  
We will observe that the error is maximized along sets where the solution is not as regular 
and along the boundary due to the auxiliary boundary condition $\Delta u^\e = 0$.  
The measured error does not appear to correspond to sets where the coefficient matrix $A$ is not 
as well-behaved.  
We also observe better rates of convergence with respect to $\e$ than the rates 
guaranteed by Theorem~\ref{thm_eps_rates}.  
Numerically we observe linear convergence in the $L^2$-norm.

\medskip
{\bf Test 1.}
Consider \eqref{eqn1.1} with $A$ defined by \eqref{numerical_testA}, $\Omega = (-2,2) \times (-2,2)$, and solution 
$u(x,y) = \frac{1}{6} \bigl| x \bigr|^3 \cos (y)$, where $f$ and $g$ are chosen accordingly.
The given test problem has a solution in $H^3(\Omega)$, 
where the third-order partial derivative with respect to $x$ is discontinuous 
along the line $x = 0$.  
The results for varying $\e$ can be found in Table~\ref{HJB_ex1_1_table}.
We can also see from Figure~\ref{HJB_ex1_1_plot} that the error is largest along the line $x = 0$, as expected.

\begin{table}[htb] 
\vspace{0.4in}
\caption{Rates of convergence for Test 1 using the vanishing moment method. 
}
\begin{center}
\scalebox{0.85}{
\begin{tabular}{| c | c | c | c | c | c | c |} 
	\hline
	$\e$ & $\left\| u^\e_h - u \right\|_{L^2(\Omega)}$ & Order & 
		$\left\| \nabla \left( u^\e_h - u \right) \right\|_{L^2(\Omega)}$ & Order & 
		$\left\| \Delta \left( u^\e_h - u \right) \right\|_{L^2(\Omega)}$ & Order  \\ 
		\hline
	4e-2 & 9.44e-3 & & 2.25e-2 & & 2.55e-1 & \\
		\hline
	2e-2 & 4.89e-3 & 0.95 & 1.32e-2 & 0.76 & 2.15e-1 & 0.24 \\
		\hline		
	1e-2 & 2.50e-3 & 0.96 & 7.84e-3 & 0.76 & 1.82e-1 & 0.24 \\
		\hline
	5e-3 & 1.27e-3 & 0.98 & 4.64e-3 & 0.76 & 1.54e-1 & 0.25 \\
		\hline
	2.5e-7 & 1.65e-7 & & 2.09e-5 & & 3.80e-3 & \\
		\hline
\end{tabular}
}
\end{center}
\label{HJB_ex1_1_table}
\end{table}

\begin{figure}[htb]
\centering
\subfloat[$u^\e_h$.]{
\includegraphics[scale=0.22]{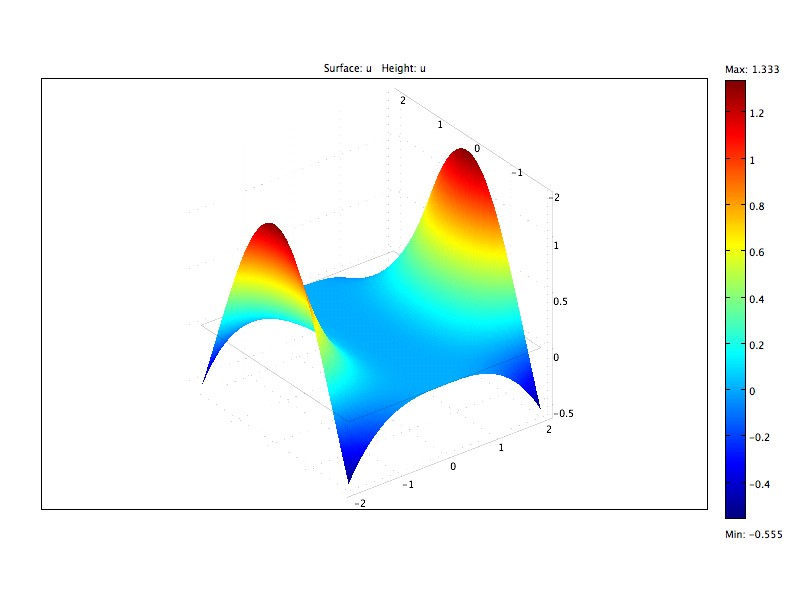}
}
\subfloat[$\Delta \left( u - u^\e_h \right)$.]{
\includegraphics[scale=0.22]{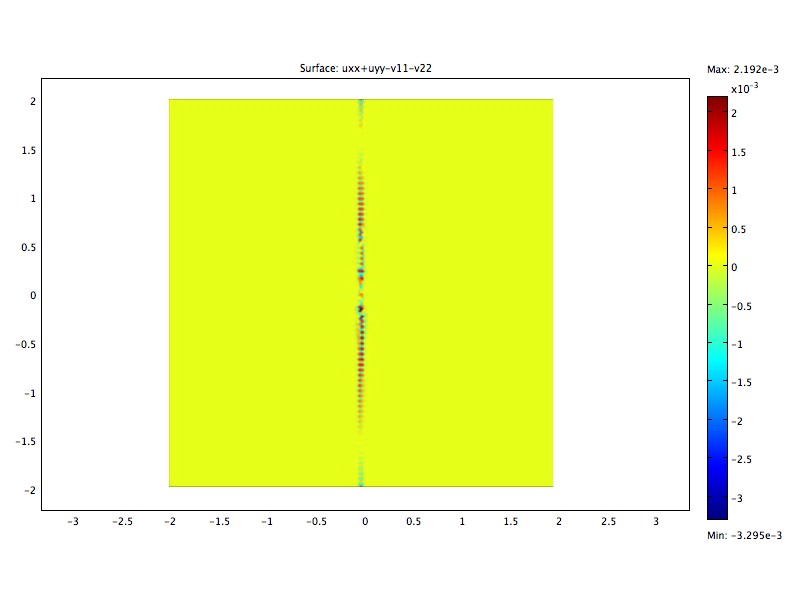}
}
\caption{
Computed solution and corresponding error for Test 1 using the vanishing moment method. 
} 
\vspace{0.35in}
\label{HJB_ex1_1_plot}
\end{figure}

\medskip
{\bf Test 2.}
Consider \eqref{eqn1.1} with $A$ defined by \eqref{numerical_testA}, $\Omega = (-2,2) \times (-2,2)$, and solution 
$u(x,y) = \frac{1}{2} x \bigl| x \bigr| \cos (y)$, where $f$ and $g$ are chosen accordingly.
The given test problem has a solution in $H^2(\Omega)$, where the second-order derivative is discontinuous 
along the line $x = 0$.  
The results for varying $\e$ can be found in Table~\ref{HJB_ex1_2_table}.  
Figure~\ref{HJB_ex1_2_plot} shows that the error is once again largest along the line $x = 0$.  The figure also 
shows the boundary layer due to the high-order auxiliary boundary condition.

\begin{table}[htb] 
\vspace{0.4in}
\caption{Rates of convergence for Test 2 using the vanishing moment method. 
}
\begin{center}
\scalebox{0.85}{
\begin{tabular}{| c | c | c | c | c | c | c |} 
	\hline
	$\e$ & $\left\| \ue_h - u \right\|_{L^2(\Omega)}$ & Order & 
		$\left\| \nabla \left( \ue_h - u \right) \right\|_{L^2(\Omega)}$ & Order & 
		$\left\| \Delta \left( \ue_h - u \right) \right\|_{L^2(\Omega)}$ & Order  \\ 
		\hline
	4e-2 & 4.30e-3 & & 2.26e-2 & & 3.99e-1 & \\
		\hline
	2e-2 & 2.30e-3 & 0.90 & 1.42e-2 & 0.67 & 3.38e-1 & 0.24 \\
		\hline		
	1e-2 & 1.22e-3 & 0.91 & 8.79e-3 & 0.69 & 2.86e-1 & 0.24 \\
		\hline
	5e-7 & 2.25e-4 & & 9.09e-4 & & 1.07e-1& \\
		\hline
\end{tabular}
}
\end{center}
\label{HJB_ex1_2_table}
\end{table}

\begin{figure}[htb]
\centering
\subfloat[$\ue_h$.]{
\includegraphics[scale=0.22]{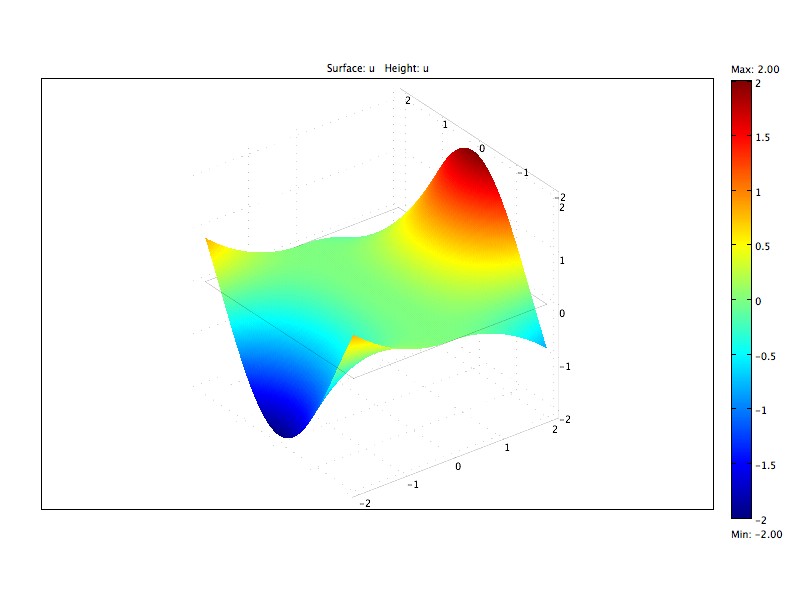}
}
\subfloat[$\Delta \left( u - \ue_h \right)$.]{
\includegraphics[scale=0.22]{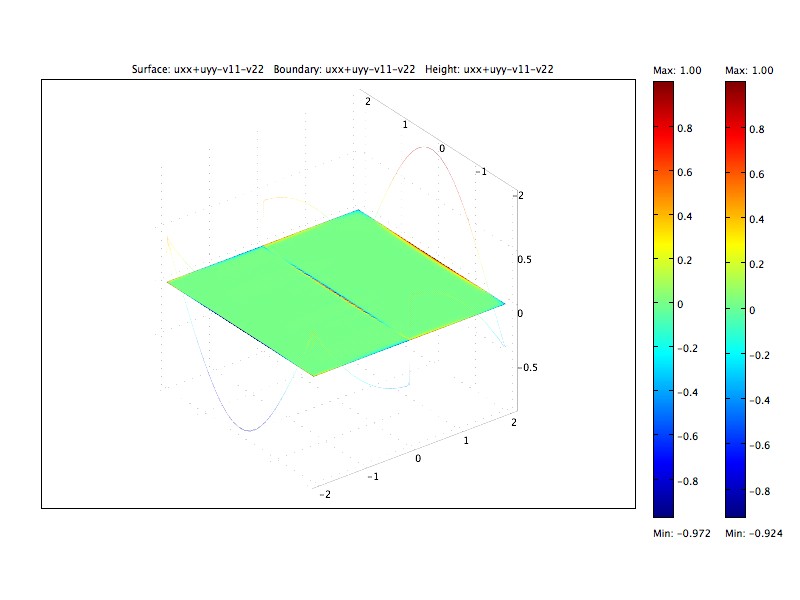}
}
\caption{
Computed solution and corresponding error for Test 2 using the vanishing moment method. 
} 
\vspace{0.3in}
\label{HJB_ex1_2_plot}
\end{figure}

\medskip
{\bf Test 3.}
Consider \eqref{eqn1.1} with $A$ defined by \eqref{numerical_testA}, 
$\Omega = B_2(0)$ (the ball of radius 2 centered at the origin), and solution 
$u(x,y) = (x-y)^{8/3}$, where $f$ and $g$ are chosen accordingly.
Observe that the solution is once again in $H^2(\Omega)$, 
where the second order derivatives have a cusp along the line $x=y$.   
The results for varying $\e$ can be found in Table~\ref{HJB_ex1_3_table}.
We can also see the finite element method does not converge to $u$ when using $\e = 0$, 
which verifies the fact that trivial $H^2$ conforming finite element methods 
do not work in general for second order linear problems 
of non-divergence form even when approximating an $H^2$ solution.  
The plot of an approximation can be found in Figure~\ref{HJB_ex1_3_plot}.

It should be noted that for many biharmonic problems with the simply supported boundary condition, approximating the circular domain $\W$ by a sequence of inscribed polygonal domains $\{\W_h\}$, as we have done in this numerical test, can lead to the so-called Babu\v{s}ka's paradox \cite{Babuska63,Davini2003} in the sense that the limiting solution may not satisfy the simply supported boundary condition on the circular boundary but, instead, satisfies the clamped boundary condition. As a result,  problems with this paradox fail to converge (due to the failed convergence of the simply supported boundary condition) as the approximate polygonal domains $\{\W_h\}$ converge in measure to the curved boundary domain $\W$.  However, we observe from this numerical test that the application of the VMM with an artificial simply supported boundary condition does not 
suffer the Babu\v{s}ka's paradox, as seen from the error $u-u_h^\e$;
this is largely because we are driving $\e\to 0$ in tandem with $h$.  
Also, the flexibility of the VMM allows alternate choices for the auxiliary boundary condition such as a clamped BC to be used because its convergence does not depend
on the choice of the artificial (higher order) boundary condition.  
As such, the VMM provides a built-in mechanism to completely avoid the difficulties associated with the Babu\v{s}ka's paradox. 

\begin{table}[htb] 
\caption{Rates of convergence for Test 3 using the vanishing moment method. 
}
\begin{center}
\scalebox{0.85}{
\begin{tabular}{| c | c | c | c | c | c | c |} 
	\hline
	$\e$ & $\left\| \ue_h - u \right\|_{L^2(\Omega)}$ & Order & 
		$\left\| \nabla \left( \ue_h - u \right) \right\|_{L^2(\Omega)}$ & Order & 
		$\left\| \Delta \left( \ue_h - u \right) \right\|_{L^2(\Omega)}$ & Order  \\ 
		\hline
	5e-3 & 1.98e-2 & & 1.22e-1 & & 4.27 & \\
		\hline
	2.5e-3 & 1.01e-2 & 0.98 & 7.32e-2 & 0.73 & 3.59 & 0.25 \\
		\hline		
	1.25e-3 & 5.08e-3 & 0.98 & 4.40e-2 & 0.73 & 3.03 & 0.24 \\
		\hline
	0 & 5.40e3 & & 1.83e6 & & 4.99e8 & \\
		\hline
\end{tabular}
}
\end{center}
\label{HJB_ex1_3_table}
\end{table}

\begin{figure}[htb]
\vspace{0.2in}
\centering
\includegraphics[scale=0.27]{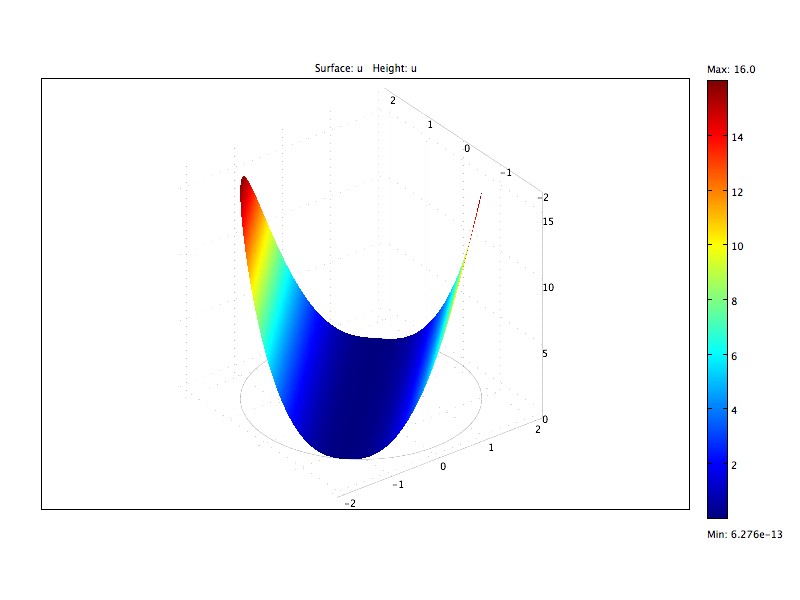}

\caption{
Computed solution for Test 3 using the vanishing moment method. 
} 
\vspace{0.3in}
\label{HJB_ex1_3_plot}
\end{figure}

\medskip
{\bf Test 4.} 
Consider \eqref{eqn1.1} with $\Omega = (-2,2) \times (-2,2)$, 
\[
	A(x,y) = \frac{16}{9} \left[
	\begin{array}{ccc}
		x^{2/3} & \, & 
			-x^{1/3} y^{1/3} \\
		-x^{1/3} y^{1/3}  & \, & 
			y^{2/3}
	\end{array}
		\right] , 
\]  
and solution $u(x,y) = x^{4/3} - y^{4/3} \in H^1(\Omega)$, where $f$ and $g$ chosen accordingly.
For the above example, we have $u \notin H^2(\Omega)$ and $A$ is not uniformly elliptic.  
Thus, $u$ is only a viscosity solution, not a strong solution.
We see in Table~\ref{HJB_ex1_4_table} that the vanishing moment method appears to be working, although 
with unknown deteriorated rates of convergence.  
From Figure~\ref{HJB_ex1_4_plot}, we see that the finite element method with Argyris finite element space 
does not work for the given example.  
Thus, we can see that the vanishing moment method has strong potential for approximating 
more general 
second order problems that are understood in the viscosity solution framework.

\begin{table}[htb] 
\caption{Rates of convergence for Test 4 using the vanishing moment method. 
}
\begin{center}
\scalebox{0.85}{
\begin{tabular}{| c | c | c | c | c |} 
	\hline
	$\e$ & $\left\| \ue_h - u \right\|_{L^2(\Omega)}$ & Order & 
		$\left\| \nabla \left( \ue_h - u \right) \right\|_{L^2(\Omega)}$ & Order \\
		\hline
	2e-4 & 6.89e-2 & & 3.94e-1 & \\
		\hline
	1e-4 & 6.10e-2 & 0.18 & 3.53e-1 & 0.16 \\
		\hline		
	5e-5 & 5.46e-2 & 0.16 & 3.15e-1 & 0.17 \\
		\hline
\end{tabular}
}
\end{center}
\label{HJB_ex1_4_table}
\end{table}

\begin{figure}[htb]
\centering
\subfloat[$\ue_h$.]{
\includegraphics[scale=0.22]{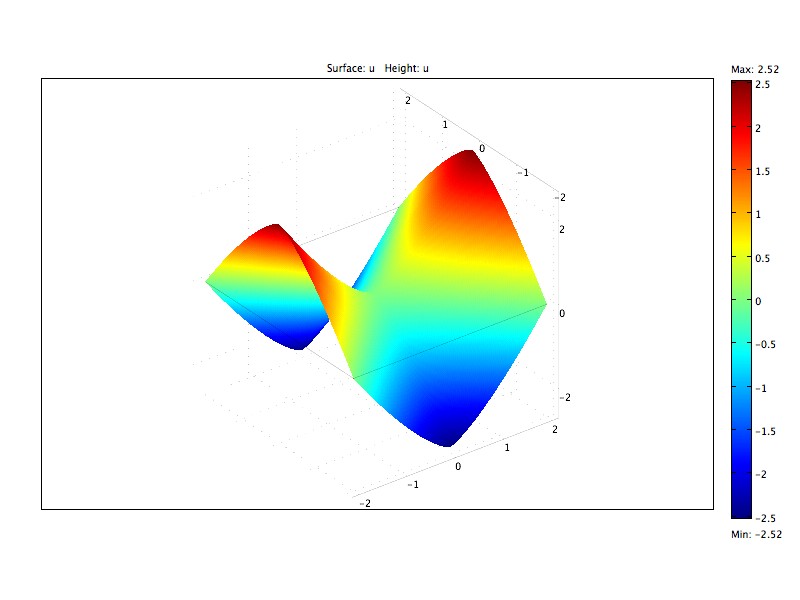}
}
\subfloat[$\ue_h$ for $\e = 0$.]{
\includegraphics[scale=0.22]{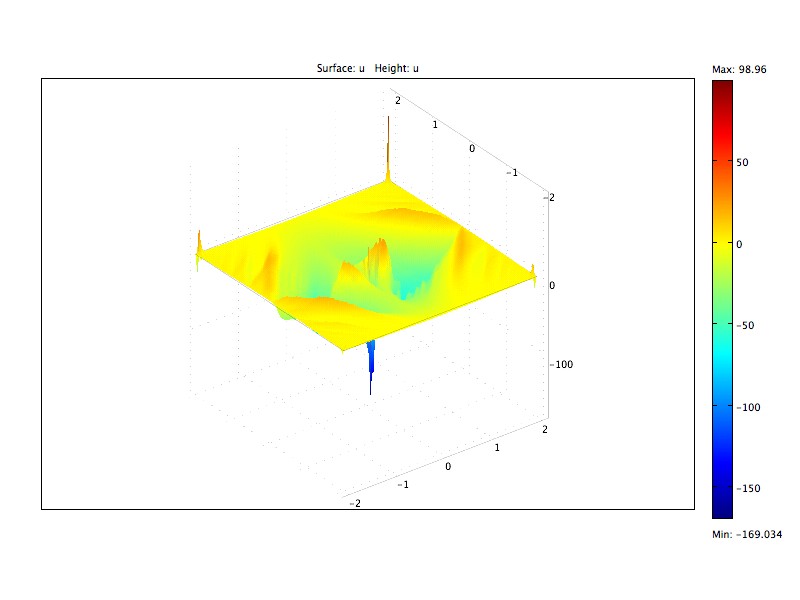}
}
\caption{
Computed solutions for Test 4 using the vanishing moment method. 
} 
\vspace{0.3in}
\label{HJB_ex1_4_plot}
\end{figure}

\bibliographystyle{abbrv}

\end{document}